\documentclass[10pt]{amsart}
\usepackage[nobysame]{amsrefs}
\usepackage[colorlinks, linkcolor=blue, citecolor=blue, urlcolor=blue]{hyperref}
\usepackage{amsthm}
\usepackage{amssymb}
\usepackage{mathtools}
\usepackage{xcolor}
\usepackage{fullpage}
\usepackage{tikz}
\usepackage[indent]{parskip}
\usepackage{longtable}
\usepackage{ytableau}
\usepackage{appendix}

\allowdisplaybreaks

\ytableausetup{centertableaux,centerboxes}

\newcommand{\vpad}[1]{%
  \ensuremath{\underset{\phantom{a}}{\overset{\phantom{a}}{#1}}}
}

\newcommand{\smallpad}[1]{%
  \scalebox{.75}{{\ensuremath{\underset{\phantom{a}}{\overset{\phantom{a}}{#1}}}}}
}

\newcommand{\mybar}[1]{
\raisebox{2pt}{$\mathsf{\overline{#1}}$}
}

\newcommand{\myvec}[1]{\left[\begin{smallmatrix} #1 \end{smallmatrix}\right]}

\DeclareMathAlphabet\mathsf{OT1}{cmss}{m}{n}

\DeclareMathAlphabet\rsfscr{U}{rsfso}{m}{n}

\tikzset{dot/.style={fill, circle, inner sep = 0pt, minimum size=5pt}}

\theoremstyle{plain}
\newtheorem{thm}{Theorem}[section]
\newtheorem{prop}[thm]{Proposition}
\newtheorem{lemma}[thm]{Lemma}

\newtheorem*{thm*}{Theorem}

\theoremstyle{definition}
\newtheorem{dfn}[thm]{Definition}
\newtheorem{alg}[thm]{Algorithm}
\newtheorem{prob}[thm]{Problem}

\theoremstyle{remark}
\newtheorem{rem}[thm]{Remark}
\numberwithin{equation}{section}

\newcommand{\R}{\mathbb{R}}
\newcommand{\Z}{\mathbb{Z}}
\newcommand{\X}{\mathbf{X}_{pq}}
\newcommand{\XY}{\mathbf{XY}\!_{pq}}
\newcommand{\M}{\rsfscr{M}\!}
\newcommand{\D}{\rsfscr{D}\hspace{-.5pt}}
\newcommand{\I}{\rsfscr{I}}

\newcommand*{\ldblbrace}{\{\mskip-5mu\{}
\newcommand*{\rdblbrace}{\}\mskip-5mu\}}
\newcommand{\mset}[1]{\ldblbrace #1 \rdblbrace}

\title{Homometric subsets of $\Z_n$ with cardinality 5: \\ classification and enumeration}

\author{William Q.~Erickson}
\address{
University of Tennessee at Chattanooga \\ 
615 McCallie Avenue\\
Chattanooga, TN 37403} 
\email{william-erickson01@utc.edu}

\author{Nicholas B.~Jones}
\address{University of North Texas\\
1155 Union Circle\\
Denton, TX 76203} 
\email{nicholas.jones@unt.edu}

\begin{document}

\begin{abstract}
    Two subsets of $\Z_n$ are said to be homometric if they have the same multiset of pairwise cyclic (i.e., Lee) distances.
    Homometric subsets necessarily have the same cardinality, say $k$.
    In this paper, for all positive integers $n$, we classify the homometric subsets of $\Z_n$ with cardinality $k=5$ (modulo cyclic shifts and reflections).
    Our classification consists of six families of homometric pairs, and one family of homometric triples.
    We also give a closed-form generating function that counts these homometric pairs and triples for all~$n$.
    The same problem for $k \leq 4$ was partially solved by Erd\H{o}s and ultimately settled by Rosenblatt--Berman (1984).
    As an immediate application of our result, one obtains an explicit criterion for the solvability of the crystallographic phase retrieval problem, in the setting of binary signals supported on five atoms.
\end{abstract}

\dedicatory{In memory of Joseph Rosenblatt (1946--2026)}



\maketitle

\section{Introduction}

\subsection*{The crystallographic binary phase retrieval problem}

X-ray crystallography is the main technique for determining the atomic structure of molecules.
In order to investigate the intrinsic structure of a crystalline substance, an X-ray beam is directed at the substance, which causes the beam to diffract.
The \emph{phase retrieval problem} is to recover the molecular structure of the substance from its diffraction pattern.
This problem admits many different variations.
In one such variant, the unknown structure of a periodic crystal (along one of the lattice vectors defining a unit cell) is given by some subset $S \subseteq \Z_n = \{0, 1, \ldots, n-1\}$.
In this setting, all atoms in the crystal are identical, and $S$ can be viewed as a binary signal (i.e., an indicator function) on $\Z_n$.
The data obtained from the corresponding diffraction pattern are the magnitudes of the values (not the values themselves) of the discrete Fourier transform $\hat{S}$; this explains the term ``phase retrieval,'' since the phase information of the signal is lost in passing from $\hat{S}$ to $|\hat{S}|$.
Thus, in this setting, the essence of the phase retrieval problem is to recover $S$ from $|\hat{S}|$.

\subsection*{Non-retrievability and homometric bracelets}

In the case of $\Z_n$, the recovery of $S$ from $|\hat{S}|$ is only ever possible up to cyclic shifts and reflections (i.e., modulo the action of the dihedral group $D_{2n}$ on the usual circular depiction of $\Z_n$); indeed, if $S$ is obtained from $T$ by a shift and/or reflection, then automatically $|\hat{S}| = |\hat{T}|$.
For this reason, it is most natural to work with the combinatorial objects known as \emph{binary bracelets} of length $n$, which can be viewed as subsets of $\Z_n$ (where the black beads represent the elements of a given subset), such that all cyclic shifts and reflections are considered the same bracelet.
We write $[S]$ to denote the bracelet that has $S$ as one of its representatives (i.e., $[S]$ is the $D_{2n}$-orbit of $S$); the \emph{weight} of $[S]$ is the number of black beads, that is, the cardinality of $S$.
A bracelet $[S]$ is said to be \emph{non-retrievable} if it cannot be recovered from $|\hat{S}|$.
This raises the key question in crystallographic phase retrieval, and in the present paper: \emph{which bracelets $[S]$ are non-retrievable?}

Non-retrievability can be understood more concretely by introducing the notion of homometry.
It is well known~\cite{BENotices}*{p.~1489} that $|\hat{S}| = |\hat{T}|$ if and only if the bracelets $[S]$ and $[T]$ are \emph{homometric}, that is, they have the same multiset of pairwise distances between their black beads.
For example, shown below are two distinct homometric bracelets of length $n=12$ and weight $5$, along with their multisets of distances:

\begin{equation}
\label{example intro}
\begin{array}{ccc}
\begin{tikzpicture}[baseline,every node/.style={scale=.75},scale=1.5]
    \node[anchor=base] at (0,0) {};
    
    \def\n{12}  
    \def\angle{360/\n}

    \draw[thick] (0,0) circle (1);

    \foreach[parse=true] \i in {0,1,...,\n-1} {
    \draw[thick, fill=white, rotate={\i*\angle}] (0,1) circle (.1);
    }

    \foreach[parse=true] \i in {0,1,3,5,6} {
    \draw[thick, fill=black, rotate={\i*\angle}] (0,1) circle (.1);
    }

    \foreach[parse=true] \i in {0,1,...,\n-1} {
    \node (\i) at ({(\i+3)*\angle}:1) {};
    }

    \draw[thick, lightgray] (0) -- (1) node [midway, above] {1};
    
    \draw[thick, lightgray] (0) -- (3) node [near end, below] {3};

    \draw[thick, lightgray] (0) -- (5) node [near start, left, xshift=2pt] {5};

    \draw[thick, lightgray] (0) -- (6) node [midway, right] {6};

    \draw[thick, lightgray] (1) -- (3) node [near start, above left, xshift=-3pt, yshift=3pt] {2};

    \draw[thick, lightgray] (1) -- (5) node [midway, left] {4};

    \draw[thick, lightgray] (1) -- (6) node [near end, left, xshift=2pt] {5};

    \draw[thick, lightgray] (3) -- (5) node [near end, below left, xshift=-3pt, yshift=-3pt] {2};

    \draw[thick, lightgray] (3) -- (6) node [near start, above] {3};

    \draw[thick, lightgray] (5) -- (6) node [midway, below] {1};
    
\end{tikzpicture}
& \qquad\qquad &
\begin{tikzpicture}[baseline,every node/.style={scale=.75},scale=1.5]
    \node[anchor=base] at (0,0) {};
    
    \def\n{12}  
    \def\angle{360/\n}

    \draw[thick] (0,0) circle (1);

    \foreach[parse=true] \i in {0,1,...,\n-1} {
    \draw[thick, fill=white, rotate={\i*\angle}] (0,1) circle (.1);
    }

    \foreach[parse=true] \i in {0,1,2,4,7} {
    \draw[thick, fill=black, rotate={\i*\angle}] (0,1) circle (.1);
    }

    \foreach[parse=true] \i in {0,1,...,\n-1} {
    \node (\i) at ({(\i+3)*\angle}:1) {};
    }

    \draw[thick, lightgray] (0) -- (1) node [midway, above] {1};
    
    \draw[thick, lightgray] (0) -- (2) node [near start, below, yshift=2pt] {2};

    \draw[thick, lightgray] (0) -- (4) node [very near end, right] {4};

    \draw[thick, lightgray] (0) -- (7) node [midway, right] {5};

    \draw[thick, lightgray] (1) -- (2) node [near start, above left, xshift=-3pt, yshift=3pt] {1};

    \draw[thick, lightgray] (1) -- (4) node [near start, right, xshift=-3pt] {3};

    \draw[thick, lightgray] (1) -- (7) node [midway, right] {6};

    \draw[thick, lightgray] (2) -- (4) node [near start, left, xshift=-5pt] {2};

    \draw[thick, lightgray] (2) -- (7) node [midway, right] {5};

    \draw[thick, lightgray] (4) -- (7) node [midway, below] {3};
    
\end{tikzpicture} \\
\text{Distances: } \mset{1,1,2,2,3,3,4,5,5,6} & & \text{Distances: } \mset{1,1,2,2,3,3,4,5,5,6}
\end{array}
\end{equation}

\noindent (These are typically called \emph{Patterson diagrams}, after~\cite{Patterson}.)
It follows that a bracelet is non-retrievable if and only if it is homometric with a different bracelet; thus the bracelets shown in~\eqref{example intro} are both non-retrievable.
In this way, classifying non-retrievable bracelets is the same as describing the equivalence classes of bracelets with respect to the homometry relation --- in particular, characterizing those equivalence classes that contain \emph{more} than one bracelet.
Following the phase retrieval literature, we call these \emph{nontrivial homometry classes}.
Note that if $[S]$ and $[T]$ are homometric, then $S$ and $T$ necessarily have the same cardinality; thus we may as well restrict our attention to those binary bracelets with some fixed weight $k$.
We therefore have the following combinatorial restatement of the crystallographic phase retrieval problem:

\begin{prob}
    \label{prob:main}
    Let $k$ be a positive integer.
    For all positive integers $n$, classify and enumerate the nontrivial homometry classes of binary bracelets with length $n$ and weight $k$.
\end{prob}

Problem~\ref{prob:main} has so far been solved only for $k \leq 4$.
In fact, for $k \leq 3$, the problem is trivial, since then every homometry class is trivial (i.e., every binary bracelet is retrievable).
For $k = 4$, a solution to Problem~\ref{prob:main} was (according to ~\cite{Rosenblatt}*{p.~337}) partially discovered by Erd\H{o}s, and was ultimately proved by Rosenblatt--Berman~\cite{Rosenblatt}*{Thm.~3.9}.
This classification of nontrivial homometry classes for $k=4$ consists of two types, which we depict below by giving the locations of the four black beads :

\[
\begin{array}{ccc}
\left\{ \quad 
\begin{tikzpicture}[baseline,every node/.style={scale=.75},scale=.75]
    \node[anchor=base] at (0,0) {};
    
    \def\n{4}  
    \def\angle{360/\n}
    
    \draw[thick, lightgray] (0,0) circle (1);

    \foreach[parse=true] \i in {0,1,...,\n-1} {
    \draw[thick, lightgray, rotate={\i*\angle}] (0,1.2) -- (0,.8);
    }


    \node at (90:1) [dot] {};
    \node at (0:1) [dot] {};
    \node (a) at (125:1) [dot] {};
    \node (b) at (-55:1) [dot] {};

    \draw [lightgray, thick, dashed] (a) -- (b);

    \node[draw,circle, fill=black, inner sep=0pt, minimum width=2pt] at (0,0) {};

    \node at (105:1.25) {$i$};
    
\end{tikzpicture}
\;\;
\begin{tikzpicture}[baseline,every node/.style={scale=.75},scale=.75]
    \node[anchor=base] at (0,0) {};
    
    \def\n{4}  
    \def\angle{360/\n}
    
    \draw[thick, lightgray] (0,0) circle (1);

    \foreach[parse=true] \i in {0,1,...,\n-1} {
    \draw[thick, lightgray, rotate={\i*\angle}] (0,1.2) -- (0,.8);
    }


    \node at (-90:1) [dot] {};
    \node at (0:1) [dot] {};
    \node (a) at (125:1) [dot] {};
    \node (b) at (-55:1) [dot] {};

    \draw [lightgray, thick, dashed] (a) -- (b);

    \node[draw,circle, fill=black, inner sep=0pt, minimum width=2pt] at (0,0) {};

    \node at (105:1.25) {$i$};
    
\end{tikzpicture} \quad
\right\} & \qquad  & \left\{ \quad 
\begin{tikzpicture}[baseline,every node/.style={scale=.75},scale=.75]
    \node[anchor=base] at (0,0) {};
    
    \def\n{13}  
    \def\angle{360/\n}
    
    \draw[thick, lightgray] (0,0) circle (1);

    \foreach[parse=true] \i in {0,1,...,\n-1} {
    \draw[thick, lightgray, rotate={\i*\angle}] (0,1.2) -- (0,.8);
    }


    \node at (90:1) [dot] {};
    \node at (90+\angle:1) [dot] {};
    \node at (90+4*\angle:1) [dot] {};
    \node at (90+6*\angle:1) [dot] {};
    
\end{tikzpicture}
\;\;
\begin{tikzpicture}[baseline,every node/.style={scale=.75},scale=.75]
    \node[anchor=base] at (0,0) {};
    
    \def\n{13}  
    \def\angle{360/\n}
    
    \draw[thick, lightgray] (0,0) circle (1);

    \foreach[parse=true] \i in {0,1,...,\n-1} {
    \draw[thick, lightgray, rotate={\i*\angle}] (0,1.2) -- (0,.8);
    }


    \node at (90:1) [dot] {};
    \node at (90+2*\angle:1) [dot] {};
    \node at (90+3*\angle:1) [dot] {};
    \node at (90+7*\angle:1) [dot] {};
    
\end{tikzpicture} \quad
\right\} \\[6ex]
\text{$n$ divisible by $4$}, & & \text{$n$ divisible by $13$}\\
1 \leq i < n/4 & &
\end{array}
\]

\noindent 
For $k \geq 5$, however, the homometry classes are not nearly as well behaved, leading Rosenblatt to observe that the $k \geq 5$ case seems very difficult to classify~\cite{Rosenblatt}*{p.~329}.

In this paper we solve Problem~\ref{prob:main} for $k=5$.
Our solution consists of a classification theorem (Theorem~\ref{thm:discrete}) and an enumeration theorem (Theorem~\ref{thm:gf}), which we preview below.

\begin{thm*}[Classification theorem; see also Theorem~\ref{thm:discrete}]
    Consider the set of binary bracelets of length $n$ and weight 5.
    Every nontrivial homometry class belongs to exactly one of the seven types $\mathsf{A}$--$\mathsf{G}$ in Figure~\ref{fig:k5 classification}.    
\end{thm*}

\begin{figure}
    \centering
    \input{NewPictures/k5_classification}
    \caption{Main result of the paper: a classification of nontrivial homometry classes of binary bracelets of length $n$ and weight $5$.
    We draw hash marks dividing each circle into equal arcs.
    A dot drawn on a hash mark represents a black bead whose position is fixed; all other dots represent black beads whose position varies with the parameters $i$ or $j$.
    Note that one must exclude $i$- and $j$-values that cause two black beads to coincide.
    (See Theorem~\ref{thm:discrete} for the bracelets written out in explicit coordinates, along with the complete description of valid $i$- and $j$-values.)
    If a pair of beads is attached to the centerpoint by dashed lines, then the two beads maintain a constant distance as $i$ varies.
    In Types~$\mathsf{A}$ and~$\mathsf{E}$ this distance is $n/2$; in Type~$\mathsf{B}$ it is $2n/5$ (left bracelet) and $n/5$ (right bracelet); in Type~$\mathsf{C}$ it is $n/6$; in Type~$\mathsf{D}$ it is $n/3$; and in Type~$\mathsf{F}$ it is $n/4$.
    The vertical dashed line in Type~$\mathsf{A}$ (resp.,~$\mathsf{D}$) remains vertical as $j$ (resp., $i$) varies.}
    \label{fig:k5 classification}
\end{figure}

\noindent For example, the homometric pair in~\eqref{example intro} is of Type~$\mathsf{A}$, where $n=12$ with $(i,j) = (1,2)$.
As expected, our classification confirms that the homometry classes for $k=5$ are much more complicated than for $k=4$.
In particular, there is a \emph{two}-parameter family of homometric pairs (Type~$\mathsf{A}$), as well as a family of homometric \emph{triples} (Type~$\mathsf{E}$).
Despite the complicated structure of these nontrivial homometry classes, our theorem above leads us to the following closed-form generating function to count these classes for arbitrary $n$.

\begin{thm*}[Enumeration theorem; see also Theorem~\ref{thm:gf}]
    Let $h_n$ denote the number of nontrivial homometry classes of binary bracelets with length $n$ and weight 5.
    We have 
    \begin{align*}
        \sum_{n=1}^\infty h_n \, x^n &= \frac{2x^{10}}{(1-x^2)(1-x^4)^2} + \frac{x^{10}+4x^{15}}{(1-x^5)(1-x^{10})} + \frac{x^{12} + x^{18}}{(1-x^6)(1-x^{12})} + \frac{4x^{16}}{(1-x^8)^2} + \frac{4x^{20}}{1-x^{20}} \\[1ex]
        &= 3x^{10} + 3x^{12} + 6x^{14} + 5x^{15} + 10x^{16} + 14x^{18} + 22x^{20} + 20x^{22} + 31x^{24} + 10x^{25} + 30x^{26} + \cdots.
    \end{align*}
\end{thm*}

\noindent By means of this generating function, one can compute any $h_n$ by taking the coefficient of $x^n$ in the Maclaurin expansion.
For example, it takes an ordinary desktop computer just a few seconds to compute $h_{15000} = 14068747$.
In fact, by making a trivial adjustment~\eqref{gf refined} to the generating function, one computes that of these $14068747$ nontrivial homometry classes, $1249$ are triples and the rest are pairs.

The special case $k=5$ is notable for the following reason.
In general, one can replace $\Z_n$ by $\Z$ in the above discussion, replacing circular distance by the standard absolute value metric.
There certainly do exist mutually homometric subsets of $\Z$ (modulo translation and reflection);
moreover, two homometric subsets $S,T \subseteq \Z$ descend to homometric length-$n$ bracelets for any $n$, via the obvious map $\Z \longrightarrow \Z_n$ given by $x \mapsto x \: {\rm mod} \: n$.
The opposite is \emph{not} true, however; indeed, $\Z$ (unlike $\Z_n$) has no mutually homometric subsets with at most 5 elements (modulo translation and reflection); see~\cite{LemkeWerman}*{p.~9}.
Thus, $5$ is the largest value of $k$ for which the crystallographic binary phase retrieval problem depends entirely on the periodicity of $\Z_n$.
In other words, for $k \geq 6$, some of the nontrivial homometry classes in $\Z_n$ are merely ``shadows'' of those in $\Z$.

\subsection*{Overview of the paper}

In order to obtain our two main theorems (previewed above) in Section~\ref{sec:main results}, we transfer the problem from the discrete setting (for the cyclic group $\Z_n$) to a more straightforward continuous setting (for the circle group $\mathbb{T}$).
Thus, Sections~\ref{sec:5-point bracelets} and~\ref{sec:classification cts} focus on the analogous problem of homometric 5-element subsets of $\mathbb{T}$.
Throughout those sections, we study \emph{5-point bracelets}, that is, orbits of 5-point subsets of $\mathbb{T}$ under rotations and reflections.
In order to control the enormous number of possible configurations of two 5-point bracelets, we first divide these configurations into five geometric categories labeled by ordered pairs $(p,q)$.   
Then we introduce a bookkeeping device we call a \emph{$(p,q)$-difference table} (Definition~\ref{def:difference table}), along with an algorithm to obtain the minimal set of 22 tables that exhaust all possible homometric configurations.
For each of these difference tables there is a corresponding \emph{solution set} (Definition~\ref{def:X and XY}) consisting of ordered pairs of vectors that encode pairs of homometric 5-point bracelets.
By (very carefully) reparametrizing these 22 solution sets, we are able to stitch them together into a seven-type classification (see Proposition~\ref{prop:classification cts}).
Then in Section 4, we translate this classification back into the discrete setting for $\Z_n$, finally obtaining Theorems~\ref{thm:discrete} and~\ref{thm:gf}.
We conclude in Section~\ref{sec:open} with some further questions and open problems (some of which have since been solved by P.~Methawisal~\cite{Methawisal}).

\subsection*{The collision-free version and Piccard's ``theorem''}

 A closely related question in the literature was the existence of homometric sets which are \emph{collision-free}, meaning that their common multiset of distances contains no repeated distances.
 Originally, an apparent result of Sophie Piccard~\cite{Piccard39} asserted that distinct homometric subsets of the real line (up to translation and reflection) cannot be collision-free;
 however, this ``theorem'' turned out to be false, as it admits precisely two families of counterexamples~\cites{Bloom77,BekirGolomb07,YovanofGolomb98}.
 (In both families, the collision-free homometric sets have cardinality $k=6$.)
 In light of this historical interest, a referee suggested that we use our computer code to observe the possible numbers of \emph{collisions}, that is, the number of distances that repeat, in each of the seven families in our classification theorem above:

 \begin{center}
 \begin{tabular}{|c|c|}
     \hline 
     Type & Number of collisions \\ \hline
     $\mathsf{A}$ & 0, 1, 2, 3, or 4 \\ \hline
     $\mathsf{B}$ & 3 or 5 \\ \hline
     $\mathsf{C}$ & 1 or 4 \\ \hline
     $\mathsf{D}$ & 0 or 2 \\ \hline
     $\mathsf{E}$ & 1 \\ \hline
     $\mathsf{F}$ & 0 or 2 \\ \hline
     $\mathsf{G}$ & 1 \\ \hline
 \end{tabular}
 \end{center}

 \noindent In particular, for the modular version of the $k=5$ problem addressed in this paper, there are infinite families of collision-free homometric pairs occurring in Types $\mathsf{A}$, $\mathsf{D}$, and $\mathsf{F}$.

\subsection*{Crystallography background}

Since its origins in X-ray crystallography in the early 1900s, the importance of the phase retrieval problem has increased due to many additional modern applications; the related literature is therefore vast, and the references we mention here form the merest skeleton of a survey.
We first recommend the excellent exposition by Bendory--Edidin~\cite{BENotices}, in particular pages 1489--1490 explaining the periodic binary setting that is the subject of the present paper.
The references in that expository article include many very recent papers, reflecting the ongoing interest in crystallographic phase retrieval; see~\cites{BE20,BEG,BEM25,Nadimi}, for example.

As mentioned above, the classification problem we address was proposed in 1984 by Rosenblatt~\cite{Rosenblatt}, to which we refer the reader especially for the exposition in its first three sections.
In that paper and elsewhere~\cites{RosenblattSeymour} are some beautiful results giving algebraic characterizations of non-retrievable signals; at the same time, however, ``algebraic techniques are typically not very
 useful in designing and analyzing practical and efficient algorithms''~\cite{BENotices}*{p.~1494} for phase retrieval.
 (See, however, the elegant proofs in~\cite{Jed24}, where homometric triples are classified using Rosenblatt's notion of \emph{spectral units}.)
 We are thus hopeful that the concrete classification we give in this paper might ultimately be aligned with, and perhaps even shed some light upon, the elegant algebraic theory.

 For readers interested in the pioneering work in crystallographic phase retrieval, we point out Pauling--Shappell~\cite{PaulingShappell} in 1930, as well as the seminal 1944 paper by Patterson~\cite{Patterson}, which (at least implicitly) contained the two-type classification for $k=4$ later proved by Rosenblatt.
 In the 1960s,  Bullough~\cite{Bullough} investigated various general properties of homometric subsets of $\Z_n$.
 Papers such as \cites{LemkeWerman,LSS} and references therein treat the equivalent problem (called the ``beltway'' or ``turnpike'' problem, depending on whether the setting is periodic or aperiodic, respectively) of reconstructing discrete sets from the multiset of their interpoint distances.
 Erd\H{o}s took an avid interest in this and several related problems~\cites{Erdos46}.

 \subsection*{The Z-relation in microtonal music theory}

 The phase retrieval problem in the setting of this paper is equivalent to a well-known problem in microtonal music theory.
 In the language of music theory, one views $\Z_n$ as a scale dividing the octave into $n$ equally-tempered pitches; the special case $n=12$, shown above in~\eqref{example intro}, corresponds to the standard 12-pitch scale typical in Western music.
 A subset $S \subseteq \Z_n$ is a \emph{chord} (or ``pitch class set'');
 in our special case, a subset $S$ with cardinality $k=5$ is typically called a \emph{pentachord}.
 The cyclic distance between two pitches is called the \emph{interval} between them, and the multiset of cyclic distances is called the \emph{interval content} of a chord.
 Given two chords that are not obtained from one another via \emph{transposition} and/or \emph{inversion} --- that is, chords corresponding to distinct bracelets --- the two chords are said to be \emph{Z-related} if their multisets of pairwise intervals coincide --- that is, if they are homometric.
 Therefore, the main result of this paper can be restated as follows, using the language of music theory: we classify and enumerate the Z-related pentachords in an equal-temperament scale consisting of an arbitrary number of pitches.

 To the reader interested in the musical motivation, we recommend the work of Forte~\cite{Forte} (who coined the term ``Z-related''), Soderberg~\cite{Soderberg}, and more recent work such as~\cites{MGAA,MGAAA}.
 In particular, we strongly recommend the monumental dissertation~\cite{Goyette} (and the references therein), which gives a comprehensive historical and musical overview of the Z-relation along with an original mathematical treatment.
 The notion of homometry has also been adapted to rhythms rather than chords~\cite{OTT}.
 Pertaining to Problem~\ref{prob:main} in the present paper, we point out the note~\cite{JJ}, which confirmed that the classification problem was very much open for $k \geq 5$ and included some enumerative computer data in various special cases.
 Closely related to our paper are the notes by Althuis--G\"obel~\cite{AG}, which contain a few special families of Z-related pentachords.
 In particular, their Properties 3 and 4 are very special cases of our Type~$\mathsf{A}$, while their Property 6 is a special case of our Type~$\mathsf{B}$.
 After completing the present paper, we became aware of the very recent result of Jedrzejewski~\cite{Jed24}*{Thm.~6}, classifying all triples of Z-related pentachords; this coincides with our Type~$\mathsf{E}$.

 \subsection*{Acknowledgments}

 We are deeply grateful to Joe Rosenblatt (1946--2026) for his kind and helpful correspondence with us, regarding his work on this problem.
 We also thank the anonymous referees for their careful reading and their detailed improvements to this paper.

\section{Homometry and difference tables for the circle group}
\label{sec:5-point bracelets}

In this and the next section, in order to temporarily bypass the bracelet length $n$, we conduct our analysis not in terms of the cyclic group $\Z_n$, but rather the circle group $\mathbb{T} \cong \R / \Z$.

\subsection*{5-point bracelets and their canonical vectors}

The notion of homometry, described in the introduction, applies equally well to finite subsets of the circle group $\mathbb{T}$, up to equivalence under rotations and reflections.
To refer concisely to such an object --- that is, the continuous analogue of a binary bracelet with weight~$k$ --- we introduce the term \emph{$k$-point bracelet} as follows.

\begin{dfn}[$k$-point bracelet]
    \label{def:k bracelet}
    Let $\mathbb{T} \coloneqq [0,1)$ be the circle group, where the operation is addition modulo~$\mathbb{Z}$.
    By depicting $\mathbb{T}$ as a circle (identifying 0 and 1), we obtain a natural action on $\mathbb{T}$ by the orthogonal group $\operatorname{O}(2)$ consisting of all rotations and reflections of the circle. 
    By a \emph{$k$-point bracelet}, we mean the $\operatorname{O}(2)$-orbit $[S]$ of some $k$-element subset $S \subset \mathbb{T}$.
    
    \noindent \textbf{NB:} In writing down an arbitrary representative $[s_1, \ldots, s_k]$ of a $k$-point bracelet, we view the $s_i$'s modulo $\mathbb{Z}$;
    hence the $s_i$'s may lie outside the interval $[0,1)$.
    
\end{dfn}

Let $S$ and $T$ be finite subsets of $\mathbb{T}$.
It follows from Definition~\ref{def:k bracelet} that $[S] = [T]$ if and only if there is some $a \in \mathbb{T}$ such that
\begin{equation}
    \label{same reps}
    T = \{a + s : s \in S \} \text{ or } T = \{a - s : s \in S\}.
\end{equation}
In both cases in~\eqref{same reps}, $T$ is obtained from $S$ via rotating by $a$, and in the second case this rotation is composed with the reflection $x \mapsto -x = 1-x$.

The cyclic (i.e., Lee) distance on $\mathbb{T}$ is given by 
\begin{equation}
    \label{distance T}
    d(t,u) = d_{\mathbb{T}}(t,u) \coloneqq \min\{|t-u|, \: 1 - |t-u| \} \in [0, \: 1/2].
\end{equation}
For this reason, we say that an element $t \in \mathbb{T}$ is \emph{short} if $t \leq 1/2$, and \emph{long} if $t > 1/2$.
The distinction is key to our analysis because for all $t,u \in \mathbb{T}$ with $t \geq u$, we have
\begin{equation}
\label{short long distance}
    d(t,u) = \begin{cases}
        t-u, & \text{$t-u$ is short},\\
        1-(t-u), & \text{$t-u$ is long}.
    \end{cases}
\end{equation}
Two $k$-point bracelets are said to be \emph{homometric} if they have the same multiset of pairwise distances between their elements (upon choosing any representative for each bracelet).
Homometry is clearly an equivalence relation on the set of all $k$-point bracelets.

\begin{dfn}[Homometry class]
    \label{def: homometry class cts}
    A \emph{homometry class} is an equivalence class of $k$-point bracelets with respect to the homometry relation.
    A homometry class is \emph{nontrivial} if it has cardinality greater than~1.
\end{dfn}

From now on, we specialize to $k=5$.
We begin our analysis by classifying 5-point bracelets into four types as follows.
If one draws any representative of a 5-point bracelet on a circle, and then draws all 10 chords connecting its 5 points, then this creates a circumscribed pentagon whose interior is divided into 11 regions: five \emph{outer triangles} (sharing a side with the pentagon), five \emph{inner triangles} (sharing a vertex but not a side with the pentagon), and one small \emph{central pentagon} (disjoint with the large pentagon).
In order to make these 11 regions mutually disjoint, we set the following boundary convention: an outer triangle does \emph{not} contain its boundary along its outer side; an inner triangle does \emph{not} contain its boundary along its two outer sides; and the central pentagon does \emph{not} contain its boundary at all.
(See the three shaded regions, respectively, in Figure~\ref{fig:long counts}.)
With these conventions, every point in the interior of the large pentagon lies in exactly one of these 11 regions.

\begin{dfn}[Long count]
    \label{def:long count}
    Consider a 5-point bracelet, depicted with the ten pairwise chords between its points as described above.
    The \emph{long count} of the 5-point bracelet is defined to be
    \[
    \begin{cases}
        0, & \text{center of the circle lies outside the interior of the large pentagon},\\
        1, & \text{center of the circle lies in an outer triangle},\\
        2, & \text{center of the circle lies in an inner triangle},\\
        3, & \text{center of the circle lies in the central pentagon}.
    \end{cases}
    \]
\end{dfn}

\begin{figure}[t]
    \centering
    \begin{tikzpicture}[scale=1.5, baseline]
    \node[anchor=base] at (0,0) {};
    
    \draw[thick] (0,0) circle (1);

    \node (a) at (90:1) [dot] {};
    \node (b) at (50:1) [dot] {};
    \node (c) at (5:1) [dot] {};
    \node (d) at (-15:1) [dot] {};
    \node (e) at (-70:1) [dot] {};

    \draw[thick, lightgray] (a) -- (b) -- (c) -- (d) -- (e) --(a) (a) -- (c) (a) -- (d) (b) -- (d) (b) -- (e) (c) -- (e);

    \node[draw,circle, fill=black, inner sep=0pt, minimum width=2pt] at (0,0) {};

    \node at (0,-1.4) {Long count 0};
    
\end{tikzpicture}
\hfill
\begin{tikzpicture}[scale=1.5, baseline]
    \node[anchor=base] at (0,0) {};
    
    \draw[thick] (0,0) circle (1);

    \node (a) at (90:1) [dot] {};
    \node (b) at (50:1) [dot] {};
    \node (c) at (5:1) [dot] {};
    \node (d) at (-40:1) [dot] {};
    \node (e) at (-110:1) [dot] {};

    \draw[thick, lightgray] (a) -- (b) -- (c) -- (d) -- (e) --(a) (a) -- (c) (a) -- (d) (b) -- (d) (b) -- (e) (c) -- (e);

    \fill[lightgray] (a.south) -- (intersection of a--d and b--e) -- (e) -- cycle;
    
    \draw [ultra thick] (a) -- (intersection of a--d and b--e) -- (e);
    \draw [ultra thick, dashed] (a) -- (e);

    \node[draw,circle, fill=black, inner sep=0pt, minimum width=2pt] at (0,0) {};

    \node at (0,-1.4) {Long count 1};
    
\end{tikzpicture}
\hfill
\begin{tikzpicture}[scale=1.5, baseline]
    \node[anchor=base] at (0,0) {};
    
    \draw[thick] (0,0) circle (1);

    \node (a) at (90:1) [dot] {};
    \node (b) at (0:1) [dot] {};
    \node (c) at (-50:1) [dot] {};
    \node (d) at (-110:1) [dot] {};
    \node (e) at (-150:1) [dot] {};

    \draw[thick, lightgray] (a) -- (b) -- (c) -- (d) -- (e) --(a) (a) -- (c) (a) -- (d) (b) -- (d) (b) -- (e) (c) -- (e);

    \fill[lightgray] (a.south) -- (intersection of a--c and b--e) -- (intersection of a--d and b--e) -- cycle;
    
    \draw [ultra thick, shorten <=2pt,shorten >=2pt] (intersection of a--c and b--e) -- (intersection of a--d and b--e);
    \draw [ultra thick, dashed] (a) -- (intersection of a--c and b--e) (a) -- (intersection of a--d and b--e);
    
    \node[draw,circle, fill=black, inner sep=0pt, minimum width=2pt] at (0,0) {};

    \node at (0,-1.4) {Long count 2};
    
\end{tikzpicture}
\hfill
\begin{tikzpicture}[scale=1.5, baseline]
    \node[anchor=base] at (0,0) {};
    
    \draw[thick] (0,0) circle (1);

    \node (a) at (90:1) [dot] {};
    \node (b) at (35:1) [dot] {};
    \node (c) at (-40:1) [dot] {};
    \node (d) at (-110:1) [dot] {};
    \node (e) at (-180:1) [dot] {};

    \draw[thick, lightgray] (a) -- (b) -- (c) -- (d) -- (e) --(a) (a) -- (c) (a) -- (d) (b) -- (d) (b) -- (e) (c) -- (e);

    \fill[lightgray] (intersection of a--c and b--e) -- (intersection of a--c and b--d) -- (intersection of c--e and b--d) -- (intersection of c--e and a--d) -- (intersection of a--d and b--e) -- cycle;
    
    \draw [ultra thick, dashed] (intersection of a--c and b--e) -- (intersection of a--c and b--d) -- (intersection of c--e and b--d) -- (intersection of c--e and a--d) -- (intersection of a--d and b--e) -- cycle;

    \node[draw,circle, fill=black, inner sep=0pt, minimum width=2pt] at (0,0) {};

    \node at (0,-1.4) {Long count 3};
    
\end{tikzpicture}
    \caption{Examples accompanying Definition~\ref{def:long count}, and the proof of Lemma~\ref{lemma: valid pq}.
    For long count 0, 1, or 2, this is the unique representative constructed in the proof of Lemma~\ref{lemma:unique rep}.}
    \label{fig:long counts}
\end{figure}

See Figure~\ref{fig:long counts} for examples.
(Alternatively, one could define the long count to be the minimal number of chords intersecting a line drawn from the center of the circle to the exterior of the pentagon.)
For the motivation behind the term ``long count'', see~\eqref{long in X} below.

\begin{lemma}
    \label{lemma: valid pq}
    If two distinct homometric 5-point bracelets have long counts $p$ and $q$ (assuming $p \leq q$), then $(p,q) \in \{(0,1), (0,2), (1,1), (1,2), (2,2)\}$.
\end{lemma}

\begin{proof}
    Suppose that $(p,q) = (0,0)$.
    Then in each bracelet, every pairwise difference is short, and therefore by~\eqref{short long distance} every pairwise \emph{distance} is just the pairwise \emph{difference}.
    But it is well known~\cite{LemkeWerman}*{p.~8} that 5 points on the real line are determined (up to translation and reflection) by their multiset of pairwise differences, and therefore the bracelets in question are not distinct, which is a contradiction.
    Thus $(p,q) \neq (0,0)$.

    Next suppose that $p \leq 2$ and $q=3$.
    Consider a generic bracelet $[Y]$ with long count 3, as pictured on the far right-hand side of Figure~\ref{fig:long counts}.
    In particular, note the five \emph{outer distances} (i.e., the arc lengths between two consecutive points) and the remaining five \emph{inner distances}.
    Since each inner distance is at most 1/2, it is the sum of the two consecutive outer distances between its endpoints.
    Therefore $[Y]$ satisfies the following:
    
    \begin{samepage} 
    \begin{align}
    \text{The sum of the outer distances} & = 1, \label{fact:outer sum 1}\\
        \text{the sum of the inner distances} &= 2. \label{fact:inner sum 2}
    \end{align}
    \end{samepage}
    
    \noindent Thus all ten distances in $[Y]$ sum to 3, and we note that in fact each distance is \emph{strictly} less than $1/2$.
    By homometry, the same must be true for the bracelet $[X]$ with long count $p$.
    In particular, the center of the circle for $[X]$ must lie outside the boundary of the central pentagon.
    But then in $[X]$, there is at least one inner distance that is strictly less than the sum of the two consecutive outer distances between its endpoints.
    Therefore, as opposed to $[Y]$, the ten distances in $[X]$ sum to strictly less than 3.
    It follows that $[X]$ and $[Y]$ have distinct distance multisets, and therefore are not homometric, which is a contradiction.
    Thus $(p,q) \notin \{(0,3), (1,3), (2,3)\}$.

    Finally, suppose that $(p,q) = (3,3)$.
    Then there exist distinct homometric bracelets $[X]$ and $[Y]$, both with long count 3.
    For both $[X]$ and $[Y]$, the following geometric facts are immediately apparent from the picture of a bracelet with long count 3 (far right-hand side of Figure~\ref{fig:long counts}), in addition to the facts~\eqref{fact:outer sum 1}--\eqref{fact:inner sum 2} above:
    \begin{align}
        \text{The sum of two consecutive outer distances} &< 1/2, \label{fact:two outers < 1/2} \\
        \text{the sum of two consecutive outer distances} &> \text{the nonadjacent outer distance}, \label{fact:two outers > opposite outer} \\
        \text{the sum of three consecutive outer distances} & > 1/2,    \label{fact:three > 1/2}
    \end{align}
    where on the right-hand side of~\eqref{fact:two outers > opposite outer}, the ``nonadjacent outer distance'' means the unique remaining outer distance which is not adjacent to either of the two consecutive outer distances on the left-hand side of~\eqref{fact:two outers > opposite outer}. 
    In the bracelet $[X]$, denote consecutive outer distances by $a,b,c,d,e$, which determines the inner distances $a+b, \ldots, e+a$.
    We first observe that if the outer distances in $[Y]$ were to include three inner distances of $[X]$, then we would contradict~\eqref{fact:outer sum 1}, because by~\eqref{fact:inner sum 2}--\eqref{fact:two outers < 1/2} any three inner distances in $[X]$ sum to $> 1$.
    Therefore the outer distances of $[Y]$ include at most two inner distances of $[X]$.
    We now treat each case (0, 1, and 2) separately.

    \textbf{Case 0: The outer distances of $[Y]$ include 0 inner distances of $[X]$.}
    In this case, the arrangement of the outer distances in $[Y]$ is obtained by permuting that in $[X]$.
    Modulo rotations and reflections, there are only three nontrivial permutations of the five outer distances $a, \ldots, e$:
    \begin{itemize}
        \item \textbf{Adjacent transposition $(12)$.}
        Suppose the outer distances of $[Y]$ are ordered $bacde$.
        Then the locations in $[Y]$ of the inner distances $a+b$ and $c+d$ and $d+e$ are forced.
        This leaves two ways to assign the distances $b+c$ and $a+e$:
        either $a+c = b+c$ (so that $a=b$) or $a+c = a+e$ (so that $c=e$).
        In either case, we have equality of bracelets $[X] = [Y]$, which is a contradiction.

        \item \textbf{Cycle $(123)$.}
        Suppose the outer distances of $[Y]$ are ordered $cabde$.
        Then the locations in $[Y]$ of the inner distances $a+b$ and $d+e$ are forced.
        There are three ways to assign the distance $b+c$.
        If $a+c=b+c$, then $a=b$, so that the outer distances of $[Y]$ are ordered $cbade = abced$, which is an adjacent transposition (see previous item).
        Likewise, if $b+d=b+c$, then $c=d$, giving the ordering $dabce = abced$, which is again an adjacent transposition.
        Finally, if $c+e = b+c$, then $b=e$, giving the ordering $caedb = acbde$, which is an adjacent transposition.
        Therefore, by the previous item, we have $[X] = [Y]$, a contradiction.

    \item \textbf{Cycle $(1342)$.}
    Suppose the outer distances of $[Y]$ are ordered $bdace$.
    Then one of the inner distances must equal $b+d$.
    If this distance is $a+b$ or $b+c$ or $c+d$ or $d+e$, then the resulting equality of outer distances implies that the outer distance ordering is actually an adjacent 3-cycle (see previous item).
    Thus we must have $a+e = b+d$.
    Proceeding in this way to assign all inner distances, we likewise obtain $a+d=b+c$ and $a+c=d+e$ and $a+b=c+e$, which forces $a=b=c=d$.
    Therefore we have $[X] = [Y]$, a contradiction.        
    \end{itemize}

    \textbf{Case 1: The outer distances of $[Y]$ include 1 inner distance of $[X]$.} 
    In this case, the outer distances of $[Y]$ consist of four outer distances of $[X]$, along with one inner distance of $[X]$.
    But applying~\eqref{fact:outer sum 1} to $[Y]$, we see that this inner distance of $[X]$ must equal the omitted outer distance of $[X]$.
    Thus we are actually in Case 0, so we are done.

    \textbf{Case 2: The outer distances of $[Y]$ include 2 inner distances of $[X]$.}
    Let $\mathbf{i} = \{i_1, i_2\}$ denote the pair of inner distances of $[X]$ that are outer distances of $[Y]$, and let $\mathbf{o} = \{o_1, o_2\}$ denote the pair of outer distances of $[X]$ that are \emph{not} outer distances of $[Y]$.
    By applying~\eqref{fact:outer sum 1} to both $[X]$ and $[Y]$, we have 
    \begin{equation}
        \label{iioo equation}
        i_1 + i_2 = o_1 + o_2.
    \end{equation}
    Note that the chords representing two inner distances either intersect each other in the interior of the circle, or else share an endpoint on the circle itself.
    Thus, up to relabeling by rotations and reflections, there are two possibilities for $\mathbf{i}$, where (without loss of generality) we take $i_1 = a+b$:
    \begin{itemize}
        \item Suppose $\mathbf{i} = \{a+b, \: b+c\}$.
        If $\mathbf{o} = \{a,b\}$ or $\{a,c\}$ or $\{b,c\}$, then $i_1 + i_2 > o_1 + o_2$, violating~\eqref{iioo equation}.
        If $\mathbf{o} = \{b,d\}$, then we have $a+b+c=d$, which violates~\eqref{fact:two outers < 1/2} and~\eqref{fact:three > 1/2}.
        Likewise, if $\mathbf{o} = \{d,e\}$, then we have $a+2b+c = d+e$, which violates~\eqref{fact:two outers < 1/2} and~\eqref{fact:three > 1/2}.
        If $\mathbf{o} = \{c,d\}$, then we have $a+2b = d$, which violates~\eqref{fact:two outers > opposite outer}, since $a$ and $b$ are consecutive outer distances, and $d$ is the outer distance nonadjacent to them.
        This leaves only the case where $\mathbf{o} = \{c,e\}$, meaning that the set of outer distances of $[Y]$ is $\{a,b,d,a+b,b+c\}$; but by~\eqref{fact:two outers < 1/2} and~\eqref{fact:three > 1/2}, $b+c$ cannot be adjacent to $a$ or $a+b$ or $d$, so this case is impossible.        

        \item Suppose $\mathbf{i} = \{a+b, \: c+d\}$.
        Note that  $\mathbf{o}$ must contain $e$;
        otherwise, we would have $i_1+i_2 = a+b+c+d >  o_1 + o_2 $, contradicting~\eqref{iioo equation}.
        If $\mathbf{o} = \{a,e\}$, then we have $b+c+d=e$, which violates~\eqref{fact:two outers < 1/2} and~\eqref{fact:three > 1/2}.
        This leaves only the case $\mathbf{o} = \{b,e\}$, meaning that the set of outer distances of $[Y]$ is $\{a,c,d,a+b,c+d\}$.
        To avoid violating~\eqref{fact:two outers < 1/2}, the order of the outer distances of $[Y]$ (up to cyclic shifts and reflections) is given either by $a,c,c+d,d,a+b$, or (via interchanging $a$ and $d$) by $d,c,a+c,a,b+d$.
        
        First suppose the order of the outer distances of $[Y]$ is $a,c,c+d,d,a+b$.
       We must have $b = a+c$, since putting $b$ as the sum of any other two consecutive outer distances would force $a+b > 1/2$, violating~\eqref{fact:two outers < 1/2}.
        Next, applying~\eqref{fact:outer sum 1} to both $[X]$ and $[Y]$, we obtain the equation $a+b+c+d+e = 2a + b + 2c + 2d$, which yields $e = a+c+d=b+d$; adding $a$ to both sides, we have $a+e=(a+b)+d$.
        One of the three remaining distances (i.e., $e$, $b+c$, and $d+e$) must equal the sum $d+(c+d)$.
        Since $d+e = 1-(a+b+c) = (2a+b+2c+2d)-(a+b+c) = 2d+a+c > 2d+c$, the inner distance in question cannot be $d+e$.
        Likewise, if we were to have $b+c = d+(c+d)$, then we would have $b=2d$; but starting from~\eqref{fact:outer sum 1} applied to $[Y]$, this would give $1 = 2a+b+2c+2d = 2a+2b+2c$, so that $a+b+c = 1/2$, violating~\eqref{fact:three > 1/2}.
        Therefore we must have $e = d+(c+d)$.
        This leaves two ways to assign the remaining inner distances $b+c$ and $d+e$.
        For each assignment, we solve for $(a,b,c,d,e)$ subject to the constraints~\eqref{fact:outer sum 1}--\eqref{fact:three > 1/2}:
            \begin{itemize}
                \item Putting $b+c = c+(c+d)$ and $d+e = a+(a+b)$ yields 
                \[
                \textstyle (a,b,c,d,e) = \left(a, \frac{1}{3}(1-2a), \frac{1}{3}(1-5a), a, \frac{1}{3}(1+a)\right), \quad \text{for } \frac{1}{14} < a < \frac{1}{8}.
                \]
                Comparing this to the outer distances of $[Y]$ given in order 
                \[
                \textstyle a, \: \frac{1}{3}(1-5a), \: \frac{1}{3}(1-2a), \: a, \: \frac{1}{3}(1+a),
                \]
                we would have $[X] = [Y]$, a contradiction.

                \item Putting $b+c = a+(a+b)$ and $d+e = c+(c+d)$ yields 
                \[
                \textstyle (a,b,c,d,e) =  \left(\frac{1}{11}, \frac{3}{11}, \frac{2}{11}, \frac{1}{11}, \frac{4}{11} \right).
                \]
                Comparing this to the outer distances of $[Y]$ given in order
                \[
                \textstyle \frac{1}{11}, \frac{2}{11}, \frac{3}{11}, \frac{1}{11}, \frac{4}{11},
                \]
                we would have $[X] = [Y]$, a contradiction.
            \end{itemize}
        Since each of these two solutions above has $a=d$, there is no need to check the remaining ordering of outer distances of $[Y]$.
        \end{itemize}
    Having shown that all three cases (0, 1, and 2) lead to a contradiction, we conclude that $(p,q) \neq (3,3)$.
\end{proof}

In light of Lemma~\ref{lemma: valid pq}, we may restrict our attention to 5-point bracelets with long count 0, 1, or 2.
We will identify each such bracelet with a unique 4-dimensional vector; by working directly with these \emph{canonical vectors} (Definition~\ref{def:canonical vec}), we will avoid any redundancy involving various representatives of the same bracelet.

\begin{samepage}
\begin{lemma}
\label{lemma:unique rep}
    Let $p \in \{0,1,2\}$.
    A 5-point bracelet has long count $p$ if and only if it has a representative $\{0, x_1, x_2, x_3, x_4\}$ satisfying the following inequalities:
    \begin{align}
        &x_1 < x_2 < x_3 < x_4, \label{proto-positive}\\
        &\begin{cases}
            x_4 \leq 1/2,  & p = 0,\\
            x_4 > 1/2 \textup{ and } x_3 \leq 1/2 \textup{ and } x_4 - x_1 \leq 1/2, & p=1,\\
            x_3 > 1/2 \textup{ and } x_2 < 1/2 \textup{ and } x_4 - x_1 \leq 1/2, & p=2,
        \end{cases} \label{proto-ineqs long count}\\
        & \begin{cases}
            x_1 \leq x_4 - x_3, \textup{ and if } x_1 = x_4 - x_3 \textup{ then } x_2 - x_1 \leq x_3 - x_2, & p = 0 \textup{ or }1,\\
            x_1 \leq 1-x_4, \;\;\textup{ and if } x_1 = 1 - x_4 \;\: \textup{ then } x_2 - x_1 \leq x_4 - x_3, & p=2.
        \end{cases} \label{proto-ineqs canonical ordering}
    \end{align}
    Moreover, this representative is unique.
\end{lemma}
\end{samepage}

\begin{proof}
    Consider a 5-point bracelet with long count $p$.
    This bracelet has several representatives containing the element $0$; any such representative can thus be written in the form $\{0, x_1, x_2, x_3, x_4\}$ with the $x_i$'s ordered according to~\eqref{proto-positive}.
    We must show that there is exactly one such representative satisfying~\eqref{proto-ineqs long count}--\eqref{proto-ineqs canonical ordering}.
    The proof follows almost immediately from the geometry of the various long counts $p$, so we will work with the depiction of $\mathbb{T}$ as a circle with $0$ on top and with values increasing clockwise.
    (See Figure~\ref{fig:long counts}.)

    Suppose $p=0$.
    Then by Definition~\ref{def:long count}, the center of the circle lies outside the interior of the pentagon formed by $0$ and the $x_i$'s.
    Therefore there are at most two ways to place the $x_i$'s such that they are all on the right-hand half of the circle (including the bottom point 1/2), thereby satisfying~\eqref{proto-ineqs long count}.
    (See the $p=0$ example in Figure~\ref{fig:long counts}.)
    In particular, these two ways are reflections of each other, via the reflection that interchanges $0$ and $x_4$; the condition~\eqref{proto-ineqs canonical ordering}, specifying the relative edge lengths of the pentagon, is true for exactly one of these reflections.

    Suppose $p=1$.
    Then the center of the circle lies in an outer triangle.
    Therefore there are at most two ways (again via the reflection $0 \leftrightarrow x_4$) to place the $x_i$'s such that this triangle has outer vertices at $0$ and $x_4$.
    (See the $p=1$ example in Figure~\ref{fig:long counts}.)
    In either such placement, the center of the circle must lie strictly to the right of the chord between $0$ and $x_4$, and weakly to the left of the chord between $0$ and $x_3$, and weakly above the chord between $x_1$ and $x_4$; thus~\eqref{proto-ineqs long count} is satisfied.
    The edge length condition~\eqref{proto-ineqs canonical ordering} is true for exactly one of the two reflections.

    Suppose $p=2$.
    Then the center of the circle lies in an inner triangle.
    There are at most two ways (via reflecting the circle horizontally) to place the $x_i$'s such that 0 is the top vertex of this inner triangle.
    (See the $p=2$ example in Figure~\ref{fig:long counts}.)
    In either such placement, the center of the circle must lie strictly to the left of the chord between $0$ and $x_3$, strictly to the left of the chord between $0$ and $x_2$, and weakly above the chord between $x_1$ and $x_4$; thus~\eqref{proto-ineqs long count} is satisfied.
    The edge condition~\eqref{proto-ineqs canonical ordering} is true for exactly one of the two horizontal reflections.

    The converse follows immediately from the geometric argument above.
\end{proof}

\begin{dfn}[Canonical vector of a bracelet]
\label{def:canonical vec}
    Given a 5-point bracelet with long count $p \in \{0,1,2\}$, its \emph{canonical vector} is the vector $\mathbf{x} = (x_1, x_2, x_3, x_4)$ whose coordinates are given by the unique representative specified in Lemma~\ref{lemma:unique rep}.
\end{dfn}

\subsection*{Difference tables}

Because our immediate goal is to characterize homometric pairs $\{[X], [Y]\}$, we will treat $[X]$ as an unknown, meaning that the coordinates $x_i$ of its canonical vector $\mathbf{x} = (x_1, x_2, x_3, x_4)$ should be treated as variables.
We set the following shorthand for the pairwise differences between the $x_i$'s (where $x_0 \coloneqq 0$):
\begin{align}
\boxed{\mathsf{ij}} &\coloneqq x_j - x_i, \nonumber \\
\boxed{\mathsf{\overline{ij}}} &\coloneqq 1 - \boxed{\mathsf{ij}}, \text{ so that } \boxed{\mathsf{\overline{\overline{ij}}}} = \boxed{\mathsf{ij}}, \nonumber \\
D & \coloneqq \left\{\boxed{\mathsf{ij}} : 0 \leq i < j \leq 4 \right\} = \left\{\boxed{\mathsf{01}},\boxed{\mathsf{02}},\boxed{\mathsf{03}},\boxed{\mathsf{04}},\boxed{\mathsf{12}},\boxed{\mathsf{13}},\boxed{\mathsf{14}},\boxed{\mathsf{23}},\boxed{\mathsf{24}},\boxed{\mathsf{34}} \right\}. \label{D}
\end{align}
The term ``long count'' is now justified by the fact that the long count of $[X]$ equals the number of long elements of $D$.
We say that an element of $D$ is \emph{$p$-long} if that element is long when $[X]$ has long count $p$.
Concretely, by~\eqref{proto-ineqs long count}, we have the following:
\begin{equation}
    \label{long in X}
    \text{The $p$-long elements of $D$ are } \begin{cases}
        \text{none}, & p = 0,\\
        \boxed{\mathsf{04}}, & p = 1, \\
        \text{$\boxed{\mathsf{03}}$ and $\boxed{\mathsf{04}}$}, & p = 2.
    \end{cases}
\end{equation}
Consider the following staircase arrangement of the elements of $D$:
\begin{equation}
    \label{table X}
    \begin{ytableau}[\mathsf]
    {01} & {02} & {03} & {04} \\
    \none & {12} & {13} & {14} \\
    \none & \none & {23} & {24}\\
    \none & \none & \none & {34}
\end{ytableau}
\end{equation}
Note that the entries in the top row of~\eqref{table X} recover the canonical vector $\mathbf{x}$, since $\boxed{\mathsf{0j}} = x_j$.
Moreover, since $x_j - x_i = (x_{i+1} - x_i) + (x_{i+2}-x_{i+1}) + \cdots + (x_j - x_{j-1})$, each off-diagonal entry is the sum of the diagonal entries lying weakly southwest of it:
\begin{equation}
    \label{sums ID}
    \boxed{\mathsf{ij}} = \sum_{d = i}^{j-1} \boxed{\mathsf{d, d+1}} \qquad \text{for all six pairs $(i,j)$ such that $j-i > 1$}.
\end{equation}

\begin{dfn}[Difference tables]
    \label{def:difference table}
    Let $(p,q) \in \{(0,1), (0,2), (1,1), (1,2), (2,2)\}$.
    Each permutation $\pi$ of the set $D$ gives rise to a unique \emph{$(p,q)$-difference table}, obtained from~\eqref{table X} as follows:
    \begin{itemize}
    \item Permute the entries in~\eqref{table X} via $\pi$.
    \item Place a bar over each $p$-long entry.
    \item Place a bar over the entry in each position corresponding to a $q$-long entry in~\eqref{table X}.
\end{itemize}
We define
\[
\D(p,q) \coloneqq \Big\{ \text{$(p,q)$-difference tables} \Big\},
\]
which contains $10!$ elements (one for each permutation $\pi$ of $D$).
Given a difference table $\Delta \in \D(p,q)$, we write 
\begin{align*}
    \boxed{\Delta(\mathsf{ij})} & \coloneqq \text{ the entry of $\Delta$ in the position corresponding to $\boxed{\mathsf{ij}}$ in~\eqref{table X}},\\
    \mathbf{y}(\Delta) & \coloneqq \left(\:\boxed{\Delta(\mathsf{01})}, \boxed{\Delta(\mathsf{02})}, \boxed{\Delta(\mathsf{03})}, \boxed{\Delta(\mathsf{04})}\:\right) = \text{top row of $\Delta$}.
\end{align*}
\end{dfn}

The idea is to view each $(p,q)$-difference table $\Delta$ as a generic pair of homometric bracelets with long counts $p$ and $q$.
In particular, the bracelet with long count $p$ is assumed to have canonical vector $\mathbf{x}$, and the bracelet with long count $q$ is assumed to have canonical vector $\mathbf{y}(\Delta)$.
In light of the geometric constraints on canonical vectors (see Lemma~\ref{lemma:unique rep}), the ultimate goal will  be to characterize all pairs $(\mathbf{x}, \mathbf{y}(\Delta)) \in \R^4 \times \R^4$ which truly give canonical vectors of 5-point bracelets.
This motivates the following definition.
(See Appendix~\ref{app:example} for a full example in which we construct a difference table and implement the following definition of its \emph{solution set}.)

\begin{dfn}[Solution set of a difference table]
    \label{def:X and XY}
    Fix $(p,q) \in \{(0,1), (0,2), (1,1), (1,2), (2,2)\}$, and consider the following linear system of equations and inequalities in the variable $\mathbf{x} = (x_1,x_2,x_3,x_4)$:
\begin{align}
    & 0 < x_1 < x_2 < x_3 < x_4, \label{positive}\\[1ex]
    &\begin{cases}
        \boxed{\mathsf{04}} \leq 1/2, & p=0,\\
        \boxed{\mathsf{04}} > 1/2 \text{ and } \boxed{\mathsf{03}} \leq 1/2 \text{ and } \boxed{\mathsf{14}} \leq 1/2, & p=1,\\
        \boxed{\mathsf{03}} > 1/2 \text{ and } \boxed{\mathsf{02}} < 1/2 \text{ and } \boxed{\mathsf{14}} \leq 1/2, & p=2,
    \end{cases} \label{ineqs long count p}\\[2ex]
    & \begin{cases}
        \boxed{\mathsf{01}} \leq \boxed{\mathsf{34}}, \text{ and if }\boxed{\mathsf{01}} = \boxed{\mathsf{34}} \text{ then } \boxed{\mathsf{12}} \leq \boxed{\mathsf{23}}, & p=0 \text{ or }1,\\
        \boxed{\mathsf{01}} \leq 1 - \boxed{\mathsf{04}}, \text{ and if }\boxed{\mathsf{01}} = 1 - \boxed{\mathsf{04}} \text{ then } \boxed{\mathsf{12}} \leq \boxed{\mathsf{34}}, & p = 2,
    \end{cases} \label{ineqs canonical ordering p}\\[1ex]
    & \boxed{\Delta(\mathsf{ij})} = \sum_{d = i}^{j-1} \boxed{\Delta(\mathsf{d, d+1})} \quad \text{for all six pairs $(i,j)$ such that $j-i>1$}, \label{eqs six sums}\\[1ex]
    & \text{same as~\eqref{ineqs long count p} and~\eqref{ineqs canonical ordering p}, but with each $\boxed{\mathsf{ij}}$ replaced by $\boxed{\Delta(\mathsf{ij})}$ and with $p$ replaced by $q$}, \label{same but for q}\\[1ex]
    & \begin{cases}
        \mathbf{x} \neq \mathbf{y}(\Delta), & p \neq q,\\
        \mathbf{x} < \mathbf{y} \text{ in lexicographical order,} & p = q,
    \end{cases} \label{noncongruent}
\end{align}
where the lexicographical condition means that the first nonzero coordinate of $\mathbf{x} - \mathbf{y}$ is negative.
(The purpose of this condition is to prevent both $(\mathbf{x}, \mathbf{y})$ and $(\mathbf{y}, \mathbf{x})$ from occurring in $\XY(\Delta)$.)
For $\Delta \in \D(p,q)$, define 
    \begin{align*}
        \X(\Delta) &\coloneqq \Big\{ \mathbf{x} = (x_1, x_2, x_3, x_4) \in \R^4 : \text{$\mathbf{x}$ satisfies the system \eqref{positive}--\eqref{noncongruent}} \Big\}, \\
        \XY(\Delta) &\coloneqq \Big\{ \big(\mathbf{x}, \: \mathbf{y}(\Delta) \big) : \mathbf{x} \in \X(\Delta) \Big\}.
    \end{align*}
    We call $\XY(\Delta)$ the \emph{$(p,q)$-solution set} of $\Delta$ (or just the \emph{solution set}, when the $(p,q)$ is clear from context).
\end{dfn}

\begin{lemma}
    \label{lemma:solutions}
    Let $(p,q) \in \{(0,1), (0,2), (1,1), (1,2), (2,2)\}$.
    We have a bijection
    \[
        \left\{ \begin{array}{c} 
        \textup{unordered pairs of distinct} \\ \textup{homometric 5-point bracelets} \\
        \textup{with long counts $p$ and $q$}
        \end{array} \right\} \longrightarrow \;\; \bigcup_{\mathclap{\Delta \in \D(p,q)}} \: \XY(\Delta), 
    \]
    given by mapping an unordered pair of bracelets to the unique ordered pair of its canonical vectors occurring in the right-hand side.
\end{lemma}

\begin{proof}
Let $[X]$ and $[Y]$ be distinct homometric 5-point bracelets, with canonical vectors $\mathbf{x} = (x_1, x_2, x_3, x_4)$ and $\mathbf{y} = (y_1, y_2, y_3, y_4)$, respectively.
Suppose $[X]$ has long count $p$, and $[Y]$ has long count $q$; if $p = q$, assume (without loss of generality) that $\mathbf{x} < \mathbf{y}$ with respect to lexicographical order.
We must first show that there exists some $\Delta \in \D(p,q)$ such that $(\mathbf{x}, \mathbf{y}) \in \XY(\Delta)$.

By definition of homometry, $[X]$ and $[Y]$ have the same distance multiset;
thus there is a permutation $\pi$ of $D$ such that, for all $\boxed{\mathsf{ij}} \in D$,
    \[
    d(y_i, y_j) = d(x_{i'}, x_{j'}) \text{ if } \pi\!\left(\boxed{\mathsf{ij}}\right) = \boxed{\mathsf{i'j'}},
    \]
where we set $x_0 = y_0 \coloneqq 0$.
Let $\Delta \in \D(p,q)$ be the $(p,q)$-difference table corresponding to this permutation $\pi$.
To show that $\mathbf{x} \in \X(\Delta)$, we verify each constraint~\eqref{positive}--\eqref{noncongruent}.
The constraints~\eqref{positive}--\eqref{ineqs canonical ordering p} are automatically satisfied due to~\eqref{proto-positive}--\eqref{proto-ineqs canonical ordering}.
Next, consider the construction of $\Delta$ according to Definition~\ref{def:difference table}, after having permuted the entries of~\eqref{table X} via $\pi$.
In particular, by first placing a bar over the $p$-long entries, we render all entries short (since $[X]$ has long count $p$);
consequently, the position corresponding to each $\boxed{\mathsf{ij}}$ in~\eqref{table X} contains the entry $d(y_i, y_j) = d(x_{i'}, x_{j'})$ rather than the entry $x_{j'} - x_{i'}$.
By then placing a bar in each position corresponding to a $q$-long entry in~\eqref{table X}, we guarantee that each 
    \begin{equation}  
        \label{entries are y diffs}
        \boxed{\Delta(\mathsf{ij})} = y_j - y_i
    \end{equation}
    rather than $d(y_i, y_j)$.
Therefore the relations~\eqref{sums ID} carry over to their corresponding positions in $\Delta$, and thus~\eqref{eqs six sums} is satisfied.
In particular,~\eqref{entries are y diffs} implies that $\mathbf{y} = \mathbf{y}(\Delta)$, and therefore~\eqref{same but for q} is satisfied due to~\eqref{proto-ineqs long count}--\eqref{proto-ineqs canonical ordering}.
By this fact $\mathbf{y} = \mathbf{y}(\Delta)$, our initial assumptions on $\mathbf{x}$ and $\mathbf{y}$ also imply~\eqref{noncongruent}.
Having thus established that $\mathbf{x} \in \X(\Delta)$, and that $\mathbf{y} = \mathbf{y}(\Delta)$, it follows by Definition~\ref{def:X and XY} that $(\mathbf{x}, \mathbf{y}) \in \XY(\Delta)$.
Moreover, by~\eqref{noncongruent} it is impossible that $(\mathbf{y}, \mathbf{x})$ occurs in $\XY(\Delta')$ for some $(p,q)$-difference table $\Delta'$, and so the map in the lemma is well-defined.

To show that the map is a bijection, we exhibit its inverse as follows.
Let $\Delta \in \D(p,q)$, and let $(\mathbf{x}, \mathbf{y}) \in \XY(\Delta)$, with coordinates given by $\mathbf{x} = (x_1, x_2, x_3, x_4)$ and $\mathbf{y} = (y_1, y_2, y_3, y_4)$.
By Definition~\ref{def:X and XY}, we have $\mathbf{x} \in \X(\Delta)$ and $\mathbf{y} = \mathbf{y}(\Delta)$. 
Set $X = \{0, x_1, x_2, x_3, x_4\}$ and $Y = \{0, y_1, y_2, y_3, y_4\}$.
By Lemma~\ref{lemma:unique rep} and by~\eqref{positive}--\eqref{ineqs canonical ordering p}, $[X]$ has canonical vector $\mathbf{x}$ and long count $p$.
By~\eqref{eqs six sums}, we have $0 <y_1 < y_2 < y_3 < y_4 < 1$.
The relations~\eqref{eqs six sums} also imply~\eqref{entries are y diffs} for every entry of $\Delta$; thus (using the same argument as in the first part of the proof describing the bars in Definition~\ref{def:difference table}) by~\eqref{same but for q} and the fact that $\Delta \in \D(p,q)$, $[Y]$ has canonical vector $\mathbf{y}$ and long count $q$.
Reversing the argument in the first part of the proof, we have an equality of multisets $\mset{d(x_i, x_j) : 0 \leq i < j \leq 4} = \mset{d(y_i, y_j) : 0 \leq i < j \leq 4}$, and thus $[X]$ and $[Y]$ are homometric.
Finally,~\eqref{noncongruent} guarantees that $[X] \neq [Y]$.
Hence the unordered pair $\{[X], [Y]\}$ is the preimage of $(\mathbf{x}, \mathbf{y})$.
\end{proof}

\section{Classification of homometry classes for the circle group}

\label{sec:classification cts}

The main result of this section (Proposition~\ref{prop:classification cts}) is a classification of the nontrivial homometry classes of 5-point bracelets.
In light of Lemma~\ref{lemma:solutions}, for each $(p,q)$, we seek a concrete description of the union
\begin{equation}
    \label{big union}
    \bigcup_{\mathclap{\Delta \in \D(p,q)}} \: \XY(\Delta),
\end{equation}
where $\D(p,q)$ is the set of $(p,q)$-difference tables (Definition~\ref{def:difference table}) and $\XY(\Delta)$ is the solution set of $\Delta$ (Definition~\ref{def:X and XY}).
Indeed,~\eqref{big union} is a sort of ``state space'' (via canonical vectors) for all unordered distinct homometric pairs with long counts $p$ and $q$.
We emphasize the word ``pairs'' in the previous sentence, because we do require more than a description of~\eqref{big union} in order to detect full homometry \emph{classes}, rather than just homometric \emph{pairs}.
For example, distinct homometric \emph{triples} will correspond to ordered triples $(\mathbf{x}, \mathbf{y}(\Delta), \mathbf{y}(\Delta'))$ with distinct entries, where $\mathbf{x} \in \X(\Delta) \cap \X(\Delta')$ for difference tables $\Delta \in \D(p,q)$ and $\Delta' \in \D(p,q')$.
For this reason, we truly do need to work with the difference tables themselves, rather than simply passing directly to the union~\eqref{big union}.

This, of course, raises the difficulty of working with the unwieldy number ($10! \approx 3.6 \times 10^6$) of difference tables in each $\D(p,q)$.
One should suspect, however, that many of the solution sets $\XY(\Delta)$ are in fact empty, or (more generally) are properly contained in other solution sets $\XY(\Delta')$.
It turns out that this is true of nearly \emph{all} difference tables.
To make this rigorous, we implement the following simple algorithm, which outputs a minimal set $\M(p,q) \subset \D(p,q)$ such that
\begin{equation}
    \label{minimal union}
    \bigcup_{\mathclap{\Delta \in \D(p,q)}} \: \XY(\Delta) = \: \: \bigcup_{\mathclap{\Delta \in \M(p,q)}} \:\XY(\Delta).
\end{equation}
Note that the system~\eqref{positive}--\eqref{noncongruent}, which defines each solution set $\XY(\Delta)$, consists entirely of linear equations and inequalities, and therefore is quite straightforward to program.

\begin{alg}\
    \label{alg:Mpq}

\textbf{Input:} a pair $(p,q) \in \{(0,1), (0,2), (1,1), (1,2), (2,2)\}$.

\textbf{Output:} a minimal set $\M(p,q)$ satisfying~\eqref{minimal union}.

\begin{itemize}

    \item Initialize $A \leftarrow \varnothing$.

    \item For each permutation $\pi$ of $D$ (the set  defined in~\eqref{D}):

    \begin{itemize}
        \item Construct the $(p,q)$-difference table $\Delta$ corresponding to $\pi$, as described in Definition~\ref{def:difference table}.

        \item If there exists some $\Delta' \in A$ such that $\XY(\Delta) \subseteq \XY(\Delta')$, then skip this $\pi$ and continue to the next permutation.

        \item Else, update $A \leftarrow \Big(A \setminus \{\Delta'' \in A : \XY(\Delta'') \subseteq \XY(\Delta) \} \Big) \cup \{\Delta\}$.
        
    \end{itemize}

    \item Set $\M(p,q) \coloneqq A$.
    
\end{itemize}
    
\end{alg}

Implementing Algorithm~\ref{alg:Mpq} for each $(p,q)$ returns the output shown below.\footnote{Our Mathematica notebook can be accessed at \url{https://github.com/WilliamQErickson/homometry-proof}.}
In particular, out of a total of $5 \times 10! \approx$ 18 million difference tables, we can restrict our attention to the following 22:
\begingroup
\arraycolsep=1.5pt
\begin{align*}
\M(0,1) &= \left\{
\begin{array}{ccccccccc}
\scalebox{.55}{\begin{ytableau}[\mathsf]
    {01} & {12} & {13} & \mybar{14} \\
    \none & {34} & {24} & {04} \\
    \none & \none & {23} & {03}\\
    \none & \none & \none & {02}
\end{ytableau}}_{\textstyle ,} & 
\scalebox{.55}{\begin{ytableau}[\mathsf]
    {34} & {23} & {13} & \mybar{03} \\
    \none & {01} & {02} & {04} \\
    \none & \none & {12} & {14}\\
    \none & \none & \none & {24}
\end{ytableau}}_{\textstyle ,} & 
\scalebox{.55}{\begin{ytableau}[\mathsf]
    {24} & {12} & {02} & \mybar{04} \\
    \none & {23} & {34} & {14} \\
    \none & \none & {01} & {03}\\
    \none & \none & \none & {13}
\end{ytableau}}_{\textstyle ,} &
\scalebox{.55}{\begin{ytableau}[\mathsf]
    {24} & {12} & {02} & \mybar{14} \\
    \none & {34} & {23} & {04} \\
    \none & \none & {01} & {03}\\
    \none & \none & \none & {13}
\end{ytableau}}_{\textstyle ,} &
\scalebox{.55}{\begin{ytableau}[\mathsf]
    {02} & {23} & {24} & \mybar{03} \\
    \none & {01} & {12} & {04} \\
    \none & \none & {34} & {14}\\
    \none & \none & \none & {13}
\end{ytableau}}_{\textstyle ,} &
\scalebox{.55}{\begin{ytableau}[\mathsf]
    {01} & {23} & {13} & \mybar{14} \\
    \none & {02} & {34} & {04} \\
    \none & \none & {12} & {24}\\
    \none & \none & \none & {03}
\end{ytableau}}_{\textstyle ,} &
\scalebox{.55}{\begin{ytableau}[\mathsf]
    {24} & {12} & {13} & \mybar{02} \\
    \none & {01} & {34} & {04} \\
    \none & \none & {23} & {14}\\
    \none & \none & \none & {03}
\end{ytableau}}_{\textstyle ,} &
\scalebox{.55}{\begin{ytableau}[\mathsf]
    {34} & {23} & {13} & \mybar{04} \\
    \none & {01} & {02} & {14} \\
    \none & \none & {12} & {24}\\
    \none & \none & \none & {03}
\end{ytableau}}_{\textstyle ,} &
\scalebox{.55}{\begin{ytableau}[\mathsf]
    {12} & {02} & {13} & \mybar{04} \\
    \none & {01} & {23} & {14} \\
    \none & \none & {34} & {03}\\
    \none & \none & \none & {24}
\end{ytableau}} \\[4ex]
\mathsf{A_1} & \mathsf{A_2} & \mathsf{D_1} & \mathsf{D_2} & \mathsf{D_3} & \mathsf{F_1} & \mathsf{F_2} & \mathsf{G_1} & \mathsf{G_2}
\end{array}
\right\}, \\[3ex]
\M(0,2) &= \left\{ \begin{array}{c}
\scalebox{.55}{\begin{ytableau}[\mathsf]
    {13} & {03} & \mybar{04} & \mybar{14} \\
    \none & {01} & {02} & {24} \\
    \none & \none & {12} & {23}\\
    \none & \none & \none & {34}
\end{ytableau}} \\[4ex]
\mathsf{B_1}
\end{array}
\right\},\\[3ex]
\M(1,1) &= \left\{
\begin{array}{ccccc}
\scalebox{.55}{\begin{ytableau}[\mathsf]
    {12} & {02} & {03} & \mybar{24} \\
    \none & {01} & {34} & {14} \\
    \none & \none & {23} & \mybar{04}\\
    \none & \none & \none & {13}
\end{ytableau}}_{\textstyle ,} & 
\scalebox{.55}{\begin{ytableau}[\mathsf]
    {01} & {02} & \mybar{04} & \mybar{14} \\
    \none & {12} & {34} & {24} \\
    \none & \none & {13} & {03}\\
    \none & \none & \none & {23}
\end{ytableau}}_{\textstyle ,} & 
\scalebox{.55}{\begin{ytableau}[\mathsf]
    {12} & {13} & {03} & \mybar{02} \\
    \none & {23} & {34} & {14} \\
    \none & \none & {01} & \mybar{04}\\
    \none & \none & \none & {24}
\end{ytableau}}_{\textstyle ,} &
\scalebox{.55}{\begin{ytableau}[\mathsf]
    {02} & \mybar{04} & {03} & \mybar{12} \\
    \none & {01} & {23} & {14} \\
    \none & \none & {34} & {13}\\
    \none & \none & \none & {24}
\end{ytableau}}_{\textstyle ,} &
\scalebox{.55}{\begin{ytableau}[\mathsf]
    {34} & {23} & {03} & \mybar{13} \\
    \none & {01} & {24} & {14} \\
    \none & \none & {02} & \mybar{04}\\
    \none & \none & \none & {12}
\end{ytableau}} \\[4ex]
\mathsf{A_3} & \mathsf{B_2} & \mathsf{B_3} & \mathsf{B_6} & \mathsf{F_3}
\end{array}
\right\}, \\[3ex]
\M(1,2) &= \left\{
\begin{array}{cccc}
\scalebox{.55}{\begin{ytableau}[\mathsf]
    {02} & {03} & \mybar{14} & {04} \\
    \none & {23} & {01} & {24} \\
    \none & \none & {13} & {34}\\
    \none & \none & \none & {12}
\end{ytableau}}_{\textstyle ,} & 
\scalebox{.55}{\begin{ytableau}[\mathsf]
    {02} & {03} & \mybar{14} & \mybar{24} \\
    \none & {23} & {34} & \mybar{04} \\
    \none & \none & {13} & {01}\\
    \none & \none & \none & {12}
\end{ytableau}}_{\textstyle ,} & 
\scalebox{.55}{\begin{ytableau}[\mathsf]
    {03} & {24} & \mybar{14} & {04} \\
    \none & {12} & {01} & {34} \\
    \none & \none & {13} & {02}\\
    \none & \none & \none & {23}
\end{ytableau}}_{\textstyle ,} &
\scalebox{.55}{\begin{ytableau}[\mathsf]
    {12} & \mybar{04} & \mybar{14} & \mybar{03} \\
    \none & {24} & {13} & {02} \\
    \none & \none & {01} & {34}\\
    \none & \none & \none & {23}
\end{ytableau}} \\[4ex]
\mathsf{C} & \mathsf{D_4} & \mathsf{G_3} & \mathsf{G_4}
\end{array}
\right\}, \\[3ex]
\M(2,2) &= \left\{
\begin{array}{cccc}
\scalebox{.55}{\begin{ytableau}[\mathsf]
    {24} & \mybar{03} & {04} & \mybar{02} \\
    \none & {01} & {13} & {14} \\
    \none & \none & {12} & {23}\\
    \none & \none & \none & {34}
\end{ytableau}}_{\textstyle ,} & 
\scalebox{.55}{\begin{ytableau}[\mathsf]
    {12} & \mybar{03} & {04} & \mybar{02} \\
    \none & {24} & {13} & {14} \\
    \none & \none & {01} & {23}\\
    \none & \none & \none & {34}
\end{ytableau}}_{\textstyle ,} & 
\scalebox{.55}{\begin{ytableau}[\mathsf]
    {12} & {23} & {03} & \mybar{24} \\
    \none & {34} & \mybar{04} & {14} \\
    \none & \none & {02} & {13}\\
    \none & \none & \none & {01}
\end{ytableau}} & \\[4ex]
\mathsf{B_4} & \mathsf{B_5} & \mathsf{F_4}
\end{array}
\right\}.
\end{align*}
\endgroup

\noindent 
Our labeling scheme above will ultimately align with the seven types $\mathsf{A}$--$\mathsf{G}$ in our main classification result; Type $\mathsf{E}$ is missing at this stage because it consists of homometric triples, which will require further analysis below in~\eqref{intersections}.
Although the order of taking permutations in the algorithm affects which particular tables constitute the output, nonetheless these 22 tables \emph{are} uniquely determined up to equivalence, where $\Delta$ and $\Delta'$ are considered equivalent if $\XY(\Delta) = \XY(\Delta')$.

In Appendix~\ref{app:tables}, we record the explicit solution set $\XY(\Delta)$ for each of the 22 tables $\Delta$ shown above; these are easily obtained by computer using Definition~\ref{def:X and XY}, and are parametrized in terms of certain free variable(s) $x_i$.
We have placed them in Appendix~\ref{app:tables} because we will still display them all --- after carefully reparametrizing them --- throughout the proof of Proposition~\ref{prop:classification cts} below (giving each explicit reparametrization so that it can be easily lined up with the data in Appendix~\ref{app:tables}, if the reader so desires).
It turns out that these 22 solution sets are \emph{almost} disjoint; more precisely, $\X(\Delta) \cap \mathbf{X}_{pq'}(\Delta')$ is empty for \emph{almost} all distinct $\Delta \in \M(p,q)$ and $\Delta' \in \M(p,q')$.
In fact, checking all such pairs by computer, we obtain the following three nonempty intersections:

\begin{equation}
    \label{intersections}
\begin{array}{|c|c|l|}
\hline

(\Delta, \Delta') & (p,q,q') & \vpad{\Big\{ \big(\mathbf{x}, \mathbf{y}(\Delta), \mathbf{y}(\Delta') \big) : \mathbf{x} \in \X(\Delta) \cap \mathbf{X}_{pq'}(\Delta') \Big\}} \\ \hline

(\mathsf{A_1}, \mathsf{A_2}) & (0,1,1) & \vpad{\left\{ \left( \myvec{x_1 \\ 1/6 + x_1 \\ 1/3 + x_1 \\ 1/2}_{\textstyle ,} \myvec{x_1 \\ 1/6 \\ 1/3 \\ 1/2 + x_1}_{\textstyle ,} \myvec{1/6 
 - x_1 \\ 1/6 \\ 1/3 \\ 2/3 - x_1}  \right) : 0 < x_1 \leq 1/12 \right\}} \\ \hline

(\mathsf{A_3}, \mathsf{B_3}) & (1,1,1) & \vpad{\left\{ \left( \myvec{1/10 \\ 3/10 \\ 4/10 \\ 6/10}_{\textstyle ,} \myvec{2/10 \\ 3/10 \\ 4/10 \\ 7/10}_{\textstyle ,} \myvec{2/10 \\ 3/10 \\ 4/10 \\ 7/10}  \right) \right\}} \\ \hline

(\mathsf{C}, \mathsf{D_4}) & (1,2,2) & \vpad{\left\{ \left( \myvec{3/12 \\ 4/12 \\ 5/12 \\ 8/12}_{\textstyle ,} \myvec{4/12 \\ 5/12 \\ 7/12 \\ 8/12}_{\textstyle ,} \myvec{4/12 \\ 5/12 \\ 7/12 \\ 8/12} \right) \right\}} \\ \hline

\text{All other pairs} & - & \vpad{\varnothing} \\ \hline

\end{array}
\end{equation}

\noindent 
Upon inspecting the vectors in~\eqref{intersections}, it is clear that every intersection $\X(\Delta) \cap \mathbf{X}_{pq'}(\Delta') \cap \mathbf{X}_{pq''}(\Delta'')$ for \emph{three} distinct difference tables is empty.

\begin{prop}
    \label{prop:classification cts}
    Every nontrivial homometry class of 5-point bracelets  belongs to exactly one of the following seven types:

    \normalfont

    \[
    \begin{array}{ll}
        \text{Type $\mathsf{A}$:} & \Big\{ \left[0, t, t+u, \frac{1}{2} + t - u, \frac{1}{2}\right], \;  \left[0, \frac{1}{2} + t, t+u, \frac{1}{2} + t - u, \frac{1}{2} \right] \Big\}, \\[1ex]
        & \text{for all $t,u \in \left(0, \frac{1}{4}\right)$ such that $t \neq u$, and if $u = \frac{1}{6}$ then $t = \frac{1}{12}$.} \\[2ex]

        \text{Type $\mathsf{B}$:} & \Big\{ \left[0, t, \frac{1}{5}, \frac{2}{5}, \frac{2}{5} + t \right], \;  \left[0, t, \frac{1}{5}, \frac{1}{5} + t, \frac{3}{5} \right] \Big\}, \text{ for all } t \in \left(0, \frac{1}{2}\right] \setminus \left\{\frac{1}{10}, \frac{1}{5}, \frac{3}{10}, \frac{2}{5} \right\}. \\[2ex]

        \text{Type $\mathsf{C}$:} & \Big\{ \left[0, \frac{1}{6} + t, \frac{1}{3}, \frac{1}{3} +t, \frac{2}{3} \right], \;  \left[0, \frac{1}{3}, \frac{1}{3} + t, \frac{1}{2} + t, \frac{2}{3} \right] \Big\}, \text{ for all } t \in \left(0, \frac{1}{12}\right]. \\[2ex]

        \text{Type $\mathsf{D}$:} & \Big\{ \left[t, \frac{1}{6} - t, \frac{1}{3}, \frac{1}{3} + t, \frac{1}{2} \right], \;  \left[0, t, \frac{1}{6} - t, \frac{1}{3} + t, \frac{5}{6} \right] \Big\}, \text{ for all } t \in \left(0, \frac{1}{4}\right) \setminus \left\{\frac{1}{12}, \frac{1}{6} \right\}. \\[2ex]

        \text{Type $\mathsf{E}$:} & \Big\{ \left[0, t, \frac{1}{6} + t, \frac{1}{3} + t, \frac{1}{2} \right], \;  \left[0, t, \frac{1}{6}, \frac{1}{3}, \frac{1}{2} + t \right], \; \left[0, \frac{1}{6} - t, \frac{1}{6}, \frac{1}{3}, \frac{2}{3} - t \right] \Big\}, \text{ for all } t \in \left(0, \frac{1}{12}\right). \\[2ex]

        \text{Type $\mathsf{F}$:} & \Big\{ \left[0, t, \frac{1}{8}, \frac{1}{4} + t, \frac{1}{2}\right], \;  \left[0, t, \frac{1}{8} + t, \frac{1}{4}, \frac{1}{2} + t \right] \Big\}, \text{ for all } t \in \left(0, \frac{1}{2}\right) \setminus \left\{\frac{1}{8}, \frac{1}{4}, \frac{3}{8} \right\}. \\[2ex]

        \text{Type $\mathsf{G}$:} & \phantom{\text{or }} \big\{ \left[0, \frac{1}{20}, \frac{2}{20}, \frac{6}{20}, \frac{9}{20} \right], \; \left[0, \frac{3}{20}, \frac{4}{20}, \frac{5}{20}, \frac{11}{20} \right] \big\} \\[1ex]
        & \text{or } \big\{ \left[0, \frac{1}{20}, \frac{4}{20}, \frac{7}{20}, \frac{9}{20} \right], \; \left[0, \frac{3}{20}, \frac{4}{20}, \frac{6}{20}, \frac{11}{20} \right] \big\} \\[1ex]
        & \text{or } \big\{ \left[0, \frac{4}{20}, \frac{5}{20}, \frac{7}{20}, \frac{13}{20} \right], \; \left[0, \frac{7}{20}, \frac{8}{20}, \frac{11}{20}, \frac{13}{20} \right] \big\} \\[1ex]
        & \text{or } \big\{ \left[0, \frac{2}{20}, \frac{7}{20}, \frac{8}{20}, \frac{11}{20} \right], \; \left[0, \frac{5}{20}, \frac{9}{20}, \frac{11}{20}, \frac{12}{20} \right] \big\}.
    \end{array}
    \]
    
\end{prop}

\begin{proof}
    By Lemma~\ref{lemma: valid pq} and by~\eqref{minimal union}, there is a bijection
    \begin{align}
    \label{big bijection}
    \begin{split}
        \bigcup_{(p,q)} \left(\bigcup_{\Delta \in \M(p,q)} \hspace{-2ex} \XY(\Delta) \right) &\longrightarrow \left\{ \begin{array}{ll} 
        \textup{unordered pairs of distinct} \\ \textup{homometric 5-point bracelets}
        \end{array} \right\},\\
        (\mathbf{x}, \mathbf{y}) & \longmapsto \big\{[0,x_1, x_2, x_3, x_4], \: [0, y_1, y_2, y_3, y_4] \big\}.
    \end{split}
    \end{align}
   Thus, in order to describe the codomain in~\eqref{big bijection}, we can simply inspect the 22 solution sets $\XY(\Delta)$, viewing their elements $(\mathbf{x}, \mathbf{y})$ as pairs of homometric bracelets via~\eqref{big bijection}.
   As mentioned above, the solution sets $\XY(\Delta)$ are given in Appendix~\ref{app:tables}; here, however, we will approach them sequentially by Type $\mathsf{A}$--$\mathsf{G}$ (rather than by long count $(p,q)$ as in Appendix~\ref{app:tables}), in order to line them up via~\eqref{big bijection} with the classification in this proposition.
   Crucially, for each $\XY(\Delta)$ we rewrite the free variable(s) $x_i$ in terms of a new parameter $t$ (and $u$, in Type $\mathsf{A}$), which effectively ``stitches together'' the various solution sets so as to coincide with the parametrization given in the proposition.
   We must also remember to account for the intersections~\eqref{intersections} when we encounter the difference tables $\mathsf{A_1}$, $\mathsf{A_2}$, $\mathsf{A_3}$, $\mathsf{B_3}$, $\mathsf{C}$, and $\mathsf{D_4}$, so as to avoid overlaps between types.

   Throughout the proof, given $(\mathbf{x}, \mathbf{y}) \in \XY(\Delta)$ for some $\Delta$, we will write $X \coloneqq \{0, x_1, x_2, x_3, x_4\}$ and $Y \coloneqq \{0, y_1, y_2, y_3, y_4\}$ to denote the representatives on the right-hand side of~\eqref{big bijection}.

\textbf{Type $\mathsf{A}$.}
There are three difference tables we labeled as Type~$\mathsf{A}$:

\begin{center} 
\begin{tabular}{|c|c|c|c|c|}
\hline
     $\Delta$ & $(p,q)$ & Parameters &  $(\mathbf{x}, \mathbf{y}) \in \XY(\Delta)$ & For all $t,u$: \\ \hline

    $\mathsf{A_1}$ & $(0,1)$ & $\begin{aligned}t&=x_1, \\ u&=x_2 - x_1\end{aligned}$ & \vpad{ \left( \myvec{t \\ t+u \\ 1/2 + t-u \\ 1/2}_{\textstyle ,} \myvec{t\\u\\1/2 - u \\ 1/2 + t} \right) } & $\begin{array}{c} 0 < t < u/2 < 1/8\\\text{or}\\0 < t = u/2 \leq 1/12\\ \end{array}$  \\ \hline

$\mathsf{A_2}$ & $(0,1)$ & $\begin{aligned}
    t &= 1/2 - x_3,\\
    u &= x_3 - x_2
\end{aligned}$ & \vpad{ \left( \myvec{u-t \\ 1/2 - t - u \\ 1/2 - t \\ 1/2}_{\textstyle ,} \myvec{t \\ u \\ 1/2 - u \\ 1/2 + t} \right) } &
$\begin{array}{c} 0 < u/2 < t < u < 1/4\\ \text{or} \\ 1/12 \leq t = u/2 < 1/8 \end{array}$ \\ \hline

$\mathsf{A_3}$ & $(1,1)$ & $\begin{aligned}
    t &= 1/2 + x_1 - x_3,\\
    u &= 1/2 - x_3
\end{aligned}$ & \vpad{\left( \myvec{t-u \\ 1/2 - 2u \\ 1/2 - u \\ 1/2 + t - u}_{\textstyle ,} \myvec{1/2 - t - u \\ 1/2 - 2u \\ 1/2 - u \\ 1 - t - u} \right)} &
$0 < u < t < 1/4$ \\ \hline

\end{tabular}
\end{center}

\noindent
For each $\Delta$ above, we will use the $(\mathbf{x}, \mathbf{y})$ column to check that $\{[X], [Y]\}$ is the same unordered pair of bracelets given in the proposition for Type~$\mathsf{A}$.
In light of~\eqref{same reps}, we reproduce those two bracelet representatives in the top row of the following table, then express them as rotations and/or reflections of $X$ and $Y$:
\[
\begin{array}{|c|r|r|}
\hline
     & \vpad{\left[ 0, \; t, \; t+u, \; \frac{1}{2} + t - u, \; \frac{1}{2} \right]} & \left[ 0, \; \frac{1}{2} + t, \; t+u, \; \frac{1}{2} + t - u, \; \frac{1}{2} \right] \\ \hline
   \vpad{\mathsf{A_1}}  & [X] & \left[ \left\{\left(\frac{1}{2} + t \right) - y : y \in Y \right\} \right] \\ \hline
   \vpad{\mathsf{A_2}} & \left[ \left\{\frac{1}{2} - x : x \in X \right\} \right] & \left[ \left\{\left(\frac{1}{2} + t \right) - y : y \in Y \right\} \right] \\ \hline
   \vpad{\mathsf{A_3}} & \left[ \left\{\left(\frac{1}{2} + t - u \right) - x : x \in X \right\} \right] & \left[ \left\{\left( t + u \right) + y : y \in Y \right\} \right] \\ \hline
\end{array}
\]

It remains to verify the values of the parameters $t$ and $u$ given in the proposition, namely, $t,u \in \left(0, 1/4 \right)$ such that $t \neq u$, and if $u = 1/6$ then $t = 1/12$.
Combining $\mathsf{A_1}$, $\mathsf{A_2}$, and $\mathsf{A_3}$ in the first table above, we confirm that $t,u \in (0,1/4)$ with $t \neq u$.
By the table of intersections~\eqref{intersections}, $\mathbf{X}_{0,1}(\mathsf{A_1})$ and $\mathbf{X}_{0,1}(\mathsf{A_2})$ intersect precisely when $u = 1/6$ for both $\mathsf{A_1}$ and $\mathsf{A_2}$, and when the $t$ for $\mathsf{A_1}$ equals $1/6$ minus the $t$ for $\mathsf{A_2}$.
(This common value must equal the parameter $x_1$ in~\eqref{intersections}.)
In this case, from the top row of~\eqref{intersections}, we see that $\mathbf{y}(\mathsf{A_1}) \neq \mathbf{y}(\mathsf{A_2})$ as long as $t \neq 1/12$ (for either $\mathsf{A_1}$ or $\mathsf{A_2}$; equivalently, for both of them).
Therefore, if $u=1/6$ and $t \neq 1/12$, then the corresponding bracelets for $\mathsf{A_1}$ and $\mathsf{A_2}$ form a homometry class with \emph{three} bracelets, which is therefore \emph{not} of Type $\mathsf{A}$, but rather of Type $\mathsf{E}$.
We conclude that if $u = 1/6$ in Type $\mathsf{A}$, then necessarily $t=1/12$, as claimed.
In this special case $(t,u) = (1/12, 1/6)$, we have $\mathbf{y}(\mathsf{A_1}) = \mathbf{y}(\mathsf{A_2}) = (1/12, 1/6, 1/3, 7/12)$, so the corresponding homometry class (of two bracelets) is common to both $\mathsf{A_1}$ and $\mathsf{A_2}$; this is no difficulty since our classification does not descend to the subtypes within Type $\mathsf{A}$.
Finally, we observe from~\eqref{intersections} that there is a single point of intersection between $\mathbf{XY}\!_{1,1}(\mathsf{A_3})$ and $\mathbf{XY}\!_{1,1}(\mathsf{B_3})$; we choose to keep this point $(t,u) = (1/5, 1/10)$ in Type $\mathsf{A}$,
so we will remember to delete it from Type~$\mathsf{B}$.

\textbf{Type $\mathsf{B}$.}
There are six difference tables we labeled as Type $\mathsf{B}$:

\begin{center} 
\begin{tabular}{|c|c|c|c|c|}
\hline
    $\Delta$ & $(p,q)$ & Parameter & $(\mathbf{x}, \mathbf{y}) \in \XY(\Delta)$ & For all $t$: \\ \hline

    $\mathsf{B_1}$ & $(0,2)$ & $t = x_1$ & \vpad{ \left( \myvec{t \\ 1/5 \\ 2/5 \\ 2/5 + t}_{\textstyle ,} \myvec{2/5 - t \\ 2/5 \\ 3/5 - t \\ 3/5} \right) } &
    $0 < t < 1/10$ \\ \hline

    $\mathsf{B_2}$ & $(1,1)$ & $t = x_1$ & \vpad{\left( \myvec{t \\ 1/5 \\ 1/5 + t \\ 3/5}_{\textstyle ,} \myvec{t \\ 1/5 \\ 2/5 \\ 2/5 + t} \right)} & $1/10 < t < 1/5$ \\ \hline

    $\mathsf{B_3}$ & $(1,1)$ & $t = 2/5 - x_1$ & \vpad{\left( \myvec{2/5 - t \\ 3/5 - t \\ 2/5 \\ 3/5}_{\textstyle ,} \myvec{1/5 \\ t \\ 2/5 \\ 2/5 + t} \right)} & $1/5 < t \leq 3/10$ \\ \hline

    $\mathsf{B_4}$ & $(2,2)$ & $t = x_2$ & \vpad{\left( \myvec{1/5 \\ t \\ 1/5 + t \\ 3/5}_{\textstyle ,} \myvec{3/5 - t \\ 4/5 - t \\ 3/5 \\ 1-t} \right) }  & $3/10 < t < 2/5$ \\ \hline

    $\mathsf{B_5}$ & $(2,2)$ & $t = 3/5 - x_1$ & \vpad{\left( \myvec{3/5 - t \\ 4/5 - t \\ 1-t \\ 3/5}_{\textstyle ,} \myvec{1/5 \\ t \\ 3/5 \\ 1/5 + t} \right)} & $2/5 < t < 1/2$ \\ \hline

    $\mathsf{B_6}$ & $(1,1)$ & -- & \vpad{\left( \myvec{1/10 \\ 3/10 \\ 5/10 \\ 6/10}_{\textstyle ,} \myvec{3/10 \\ 4/10 \\ 5/10 \\ 8/10} \right)} & -- \\ \hline
\end{tabular}
\end{center}

\noindent
Using the $(\mathbf{x}, \mathbf{y})$ column, we verify for each $\Delta$ that $[X]$ and $[Y]$ are the two bracelets given in the proposition (reproduced in the top row of the following table):
\[
\begin{array}{|c|r|r|}
\hline
 & \vpad{
\left[ 0, \; t, \; \frac{1}{5}, \; \frac{2}{5}, \; \frac{2}{5} + t \right]} & \left[ 0, \; t, \; \frac{1}{5}, \; \frac{1}{5} + t, \; \frac{3}{5} \right] \\ \hline
\vpad{\mathsf{B_1}} & [X] & \left[ \left\{ \frac{3}{5} - y : y \in Y \right\} \right] \\ \hline
\vpad{\mathsf{B_2}} & [Y] & [X] \\ \hline 
\vpad{\mathsf{B_3}} & [Y] & \left[ \left\{ \frac{3}{5} - x : x \in X \right\} \right] \\ \hline
\vpad{\mathsf{B_4}} & \left[ \left\{ \left(\frac{2}{5} + t \right) + y : y \in Y \right\} \right] & [X] \\ \hline
\vpad{\mathsf{B_5}} & \left[ \left\{ \left(\frac{2}{5} + t \right) + x : x \in X \right\} \right] & [Y] \\ \hline
\mathsf{B_6} & \vpad{\begin{array}{r}
    \text{Upon setting $t = \frac{1}{2}$:}\\[1ex]
    \left[ \left\{ \frac{1}{2} - x : x \in X \right\} \right]
\end{array}} & \begin{array}{r}
    \text{Upon setting $t = \frac{1}{2}$:}\\[1ex]
    \left[ \left\{ 1 - y : y \in Y \right\} \right]
\end{array} \\ \hline
\end{array}
\]
Having included $t = 1/2$ to account for $\mathsf{B}_6$, we see that the union of the $t$-intervals listed in the first table above agrees with the $t$-values given in the proposition, except for the presence of the point $t = 3/10$.
But for $t = 3/10$, we observe that
\[
(\mathbf{x}, \mathbf{y}) = \left( \myvec{1/10 \\ 3/10 \\ 4/10 \\ 6/10}_{\textstyle ,} \myvec{2/10 \\ 3/10 \\ 4/10 \\ 7/10} \right)
\]
was the point of intersection between $\mathbf{XY}\!_{1,1}(\mathsf{A_3})$ and $\mathbf{XY}\!_{1,1}(\mathsf{B_3})$ given in~\eqref{intersections}.
Since we retained this point in Type~$\mathsf{A}$ above, we delete it here from Type~$\mathsf{B}$, and therefore $t=3/10$ has been removed from the set of $t$-values in the proposition.

\textbf{Type $\mathsf{C}$.} 
There is only one difference table we labeled as Type $\mathsf{C}$:

   \begin{longtable}{|c|c|c|c|c|}
\hline
    $\Delta$ & $(p,q)$ & Parameter & $(\mathbf{x}, \mathbf{y}) \in \XY(\Delta)$ & For all $t$: \\ \hline

    $\mathsf{C}$ & $(1,2)$ & $t = x_1 - 1/6$ & \vpad{\left( \myvec{1/6 + t \\ 1/3 \\ 1/3 + t \\ 2/3}_{\textstyle ,} \myvec{1/3 \\ 1/3 + t \\ 1/2 + t \\ 2/3} \right)} & $0 < t \leq 1/12$ \\ \hline

\end{longtable}

\noindent 
From this, we see that $X$ and $Y$ are precisely the representatives
\[
\textstyle \left\{0, \; \frac{1}{6} + t, \; \frac{1}{3}, \; \frac{1}{3} +t, \; \frac{2}{3} \right\} \text{ and } \left \{0, \; \frac{1}{3}, \; \frac{1}{3} + t, \; \frac{1}{2} + t, \; \frac{2}{3} \right\}
\]
given in the proposition, and the $t$-values also match those given in the proposition.
By~\eqref{intersections}, there is a single point of intersection between $\mathbf{XY}\!_{1,2}(\mathsf{C})$ and $\mathbf{XY}\!_{1,2}(\mathsf{D_4})$; we choose to retain this point $t=1/12$ in Type $\mathsf{C}$, and so we will remember to delete it from Type $\mathsf{D}$.

\textbf{Type $\mathsf{D}$.}
There are four difference tables we labeled as Type $\mathsf{D}$:

\begin{center} 
\begin{tabular}{|c|c|c|c|c|}
\hline
     $\Delta$ & $(p,q)$ & Parameter & $(\mathbf{x}, \mathbf{y}) \in \XY(\Delta)$ & For all $t$: \\ \hline

$\mathsf{D_1}$ & $(0,1)$ & $t = 1/3 - x_2$ & \vpad{ \left( \myvec{1/6 - 2t \\ 1/3 - t \\ 1/3 \\ 1/2 - t}_{\textstyle ,} \myvec{1/6 \\ 1/6 + t \\ 1/3 - t \\ 1/2 + t} \right) } & $0 < t < 1/12$ \\ \hline

$\mathsf{D_2}$ & $(0,1)$ & $t = x_2 - 1/6$ & \vpad{ \left( \myvec{2t - 1/6\\1/6 + t\\1/6 + 2t\\1/3 + t}_{\textstyle ,} \myvec{1/6 \\ 1/3 - t \\ 1/6 + t \\ 1/2 + t} \right) } & $1/12 < t < 1/9$  \\ \hline

$\mathsf{D_3}$ & $(0,1)$ & $t = 1/6 - x_1$ & \vpad{\left( \myvec{1/6 - t \\ 1/6 \\ 1/2 - t \\ 1/3 + t}_{\textstyle ,} \myvec{1/6 \\ 1/3 - t \\ 1/6 + t \\ 1/2 + t} \right)}  & $1/9 \leq t < 1/6$\\ \hline

$\mathsf{D_4}$ & $(1,2)$ & $t = x_1$ & \vpad{\left( \myvec{t \\ 2t - 1/6 \\ 1/6 + t \\ 2/3}_{\textstyle ,} \myvec{2t - 1/6 \\ 1/6 + t \\ 1/3 + t \\ 1/6 + 2t} \right)} & $1/6 < t \leq 1/4$ \\ \hline

\end{tabular}
\end{center}
\noindent Using the $(\mathbf{x}, \mathbf{y})$ column, we verify for each $\Delta$ that $[X]$ and $[Y]$ are the bracelets given in the proposition:
\[
\begin{array}{|c|r|r|}
\hline
     & \vpad{\left[ t, \; \frac{1}{6} - t, \; \frac{1}{3}, \; \frac{1}{3} + t, \; \frac{1}{2} \right]} & \left[ 0, \; t, \; \frac{1}{6} - t, \; \frac{1}{3} + t, \; \frac{5}{6} \right] \\ \hline
   \vpad{\mathsf{D_1}}  & \left[ \left\{t + x : x \in X \right\} \right] & \left[ \left\{\frac{5}{6} + y : y \in Y \right\} \right] \\ \hline
   \vpad{\mathsf{D_2}} & \left[ \left\{\left(\frac{1}{6} -t \right) + x : x \in X \right\} \right] & \left[ \left\{\frac{5}{6} + y : y \in Y \right\} \right] \\ \hline
   \vpad{\mathsf{D_3}} & \left[ \left\{\frac{1}{2} - x : x \in X \right\} \right] & \left[ \left\{\frac{5}{6} + y : y \in Y \right\} \right] \\ \hline
   \vpad{\mathsf{D_4}} & \left[ \left\{\left(\frac{1}{6} - t \right) + y : y \in Y \right\} \right] & \left[ \left\{t - x : x \in X \right\} \right] \\ \hline
\end{array}
\]

\noindent The $t$-values in the first table above agree with those in the proposition, except for the presence of $t = 1/4$.
But for $t = 1/4$, we observe that
\[
(\mathbf{x}, \mathbf{y}) = \left( \myvec{3/12 \\ 4/12 \\ 5/12 \\ 8/12}_{\textstyle ,} \myvec{4/12 \\ 5/12 \\ 7/12 \\ 8/12} \right)
\]
was the point of intersection between $\mathbf{XY}\!_{1,2}(\mathsf{C})$ and $\mathbf{XY}\!_{1,2}(\mathsf{D_4})$ given in~\eqref{intersections}.
Since we retained this point in Type~$\mathsf{C}$ above, we delete it here from Type~$\mathsf{D}$, and therefore $t = 1/4$ has been removed from the set of $t$-values in the proposition.

\textbf{Type $\mathsf{E}$.}
As we discussed in the proof for Type $\mathsf{A}$ above, the homometry classes of Type $\mathsf{E}$ have cardinality 3, and correspond to the one-parameter intersection given in the top row of~\eqref{intersections}.
Upon setting $t = x_1$ in the top row of~\eqref{intersections}, it is immediately clear that the resulting triple $(\mathbf{x}, \mathbf{y}, \mathbf{y'})$ gives the canonical vectors of the three bracelets in Type~$\mathsf{E}$, since the vectors are obtained from the bracelet representatives by removing the 0.
The $t$-values obtained from~\eqref{intersections} are $0 < t \leq 1/12$; however, as mentioned in the proof for Type~$\mathsf{A}$, the value $t=1/12$ gives a pair in Type~$\mathsf{A}$ rather than a triple in Type~$\mathsf{E}$.

\textbf{Type $\mathsf{F}$.}
There are four difference tables we labeled as Type~$\mathsf{F}$:

\begin{center} 
\begin{tabular}{|c|c|c|c|c|}
\hline
     $\Delta$ & $(p,q)$ & Parameter &  $(\mathbf{x}, \mathbf{y}) \in \XY(\Delta)$ & For all $t$:\\ \hline

$\mathsf{F_1}$ & $(0,1)$ & $t = x_1$ & \vpad{\left( \myvec{t \\ 1/8 \\ 1/4 + t \\ 1/2}_{\textstyle ,} \myvec{t \\ 1/8 + t \\ 1/4 \\ 1/2 + t} \right)} & $0 < t < 1/8$ \\ \hline

$\mathsf{F}_2$ & $(0,1)$ & $t = 1/4 - x_1$ & \vpad{\left( \myvec{1/4 - t \\ 1/2- t \\ 3/8 \\ 1/2}_{\textstyle ,} \myvec{t \\ 1/4 \\ 1/8 + t \\ 1/2 + t} \right)} & $1/8 < t < 1/4$ \\ \hline

$\mathsf{F_3}$ & $(1,1)$ & $t = 1/4 + x_1$ & \vpad{\left( \myvec{t - 1/4 \\ 1/4 \\ 1/8 + t \\ 1/4 + t}_{\textstyle ,} \myvec{1/8 \\ t - 1/8 \\ t + 1/8 \\ 5/8} \right)} & $1/4 < t < 3/8$ \\ \hline

$\mathsf{F_4}$ & $(2,2)$ & $t = 1/8 + x_2$  & \vpad{\left( \myvec{1/8 \\ t - 1/8 \\ t + 1/8 \\ 5/8}_{\textstyle ,} \myvec{t - 1/4 \\ 1/4 \\ 1/8 + t \\ 1/4 + t} \right)} & $3/8 < t < 1/2$ \\ \hline

\end{tabular}
\end{center}

\noindent Using the $(\mathbf{x}, \mathbf{y})$ column, we verify for each $\Delta$ that $[X]$ and $[Y]$ are the bracelets given in the proposition:
\[
\begin{array}{|c|r|r|}
\hline
     & \vpad{\left[ 0, \; t, \; \frac{1}{8}, \; \frac{1}{4} + t, \; \frac{1}{2} \right]} & \left[ 0, \; t, \; \frac{1}{8} + t, \; \frac{1}{4}, \; \frac{1}{2} + t \right] \\ \hline
   \vpad{\mathsf{F_1}}  & [X] & [Y] \\ \hline
   \vpad{\mathsf{F_2}} & \left[ \left\{\frac{1}{2} - x : x \in X \right\} \right] & [Y] \\ \hline
   \vpad{\mathsf{F_3}} & \left[ \left\{ \left(\frac{1}{4} + t \right) - x : x \in X \right\} \right] & \left[ \left\{\left(\frac{1}{8} + t \right) - y : y \in Y \right\} \right] \\ \hline
   \vpad{\mathsf{F_4}} & \left[ \left\{\left(\frac{1}{4} + t \right) - y : y \in Y \right\} \right] & \left[ \left\{ \left(\frac{1}{8} + t \right) -  x : x \in X \right\} \right] \\ \hline
\end{array}
\]

\noindent The $t$-values given in the proposition are immediate from the first table above.

\textbf{Type $\mathsf{G}$.}
There are four remaining difference tables in $\bigcup_{(p,q)} \M(p,q)$:

\begin{center}
\begin{tabular}{|c|c|c|}
\hline
     $\Delta$ & $(p,q)$ & $\XY(\Delta)$ \\ \hline

$\mathsf{G_1}$ &  $(0,1)$ & \vpad{\left\{ \left( \myvec{1/20 \\ 2/20 \\ 6/20 \\ 9/20}_{\textstyle ,} \myvec{3/20 \\ 4/20 \\ 5/20 \\ 11/20} \right) \right\}} \\ \hline

$\mathsf{G_2}$ & $(0,1)$ & \vpad{\left\{ \left( \myvec{1/20 \\ 4/20 \\ 7/20 \\ 9/20}_{\textstyle ,} \myvec{3/20 \\ 4/20 \\ 6/20 \\ 11/20} \right) \right\}} \\ \hline

$\mathsf{G_3}$ & $(1,2)$ & \vpad{\left\{ \left( \myvec{4/20 \\ 5/20 \\ 7/20 \\ 13/20}_{\textstyle ,} \myvec{7/20 \\ 8/20 \\ 11/20 \\ 13/20} \right) \right\}} \\ \hline

$\mathsf{G_4}$ & $(1,2)$ & \vpad{\left\{ \left( \myvec{2/20 \\ 7/20 \\ 8/20 \\ 11/20}_{\textstyle ,} \myvec{5/20 \\ 9/20 \\ 11/20 \\ 12/20} \right) \right\}} \\ \hline

\end{tabular}
\end{center}

\noindent Each difference table above has a solution set consisting of a single point $(\mathbf{x}, \mathbf{y})$, giving the canonical vectors of one of the four bracelet pairs given in Type $\mathsf{G}$ in the proposition.
\end{proof}

\begin{rem}
    \label{rem: type G parameter}
    The four pairs in Type $\mathsf{G}$ are unified via a single discrete parameter $t$ as follows:
    \begin{equation}
        \label{Type G parameter}
    \textstyle \left\{ \left[ 0, \; \frac{5}{20}, \; \frac{8t}{20}, \; \frac{15+4t}{20}, \; \frac{5-4t}{20} \right], \; \left[ 0, \; \frac{5}{20}, \; \frac{4t}{20}, \; \frac{15+8t}{20}, \; \frac{15-4t}{20} \right] \right\}, \quad \text{for $t = 1,2,3,4$}.
    \end{equation}
    In particular, with $X$ and $Y$ as given in the Type $\mathsf{G}$ table above, we have the following:
    \[
\begin{array}{|c|c|r|r|}
\hline
     & t & \vpad{\left[ 0, \; \frac{5}{20}, \; \frac{8t}{20}, \; \frac{15+4t}{20}, \; \frac{5-4t}{20} \right]} & \left[ 0, \; \frac{5}{20}, \; \frac{4t}{20}, \; \frac{15+8t}{20}, \; \frac{15-4t}{20} \right] \\ \hline
   \vpad{\mathsf{G_1}}  & 1 & \left[ \left\{ x - \frac{1}{20} : x \in X \right\} \right] & [Y] \\ \hline
   \vpad{\mathsf{G_2}} & 2 & \left[ \left\{ x - \frac{4}{20} : x \in X \right\} \right] & \left[ \left\{ \frac{11}{20} - y : y \in Y \right\} \right] \\ \hline
   \vpad{\mathsf{G_3}} & 3 & [X] & \left[ \left\{ y - \frac{8}{20} : y \in Y  \right\} \right] \\ \hline
   \vpad{\mathsf{G_4}} & 4 & [Y] & \left[ \left\{ \frac{7}{20} - x : x \in X \right\} \right] \\ \hline
\end{array}
\]
\end{rem}

\vfill

\section{Homometry classes of binary bracelets with weight $k=5$}

\label{sec:main results}

Having classified the nontrivial homometry classes of 5-point bracelets in Proposition~\ref{prop:classification cts}, we will now translate the result from the continuous setting ($\mathbb{T}$) into the discrete setting ($\Z_n$).
This section contains the main results of the paper --- a classification theorem and an enumeration theorem --- which together constitute a solution to Problem~\ref{prob:main} in the case $k=5$.
In place of the 5-point bracelets from the previous sections, we now turn to binary bracelets with 5 black beads, in the typical sense of the combinatorics literature.

\subsection*{Restatement of the problem for $\Z_n$}

We briefly formalize the nontechnical description of Problem~\ref{prob:main} given in the introduction.
The relevant definitions are completely analogous to those given in the continuous setting of Section~\ref{sec:5-point bracelets}.

Let $n$ be a positive integer, and let $\Z_n \coloneqq \{0,1,\ldots,n-1\}$ denote the cyclic group under addition modulo~$n$.
There is a natural action on $\Z_n$ by the dihedral group $D_{2n}$, realized by arranging the elements of $\Z_n$ in order at the vertices of a regular $n$-gon.
For our purposes, we define a \emph{binary bracelet of length $n$} to be the $D_{2n}$-orbit of some subset $S = \{s_1, \ldots, s_k\} \subseteq \Z_n$.
The cardinality $k$ is called the \emph{weight} of the bracelet.
We depict this bracelet, denoted by $[S] = [s_1, \ldots, s_k]$, as a circular arrangement of $n$ beads, where the $k$ black beads correspond to the elements of $S$.
As in Definition~\ref{def:k bracelet}, we note that when writing down an arbitrary representative $[s_1, \ldots, s_k]$ of a bracelet, we view the $s_i$'s modulo $n$, and so they themselves may lie outside $\{0, \ldots, n-1\}$. 

The cyclic (i.e., Lee) distance between two elements $a,b \in \Z_n$ is given by 
\begin{equation}
    \label{distance Zn}
    d(a,b) = d_{\Z_n}(a,b) \coloneqq \min\{|a-b|, \: n-|a-b| \}.
\end{equation}
Since the action of $D_{2n}$ preserves distances, it makes sense to define the \emph{distance multiset} of a binary bracelet to be the multiset of $\binom{k}{2}$ pairwise distances between the black beads of the bracelet.
Two binary bracelets of the same length are said to be \emph{homometric} if they have the same distance multiset.
Recall the diagram~\eqref{example intro} for an example of two homometric bracelets $[0,6,7,9,11]$ and $[0, 5, 8, 10, 11]$ of length $n=12$ and weight $k=5$.
The homometry relation is clearly an equivalence relation on the set of all binary bracelets of a given length $n$.
Thus it makes sense to adopt Definition~\ref{def: homometry class cts} in this discrete setting: a \emph{homometry class} is an equivalence class with respect to the homometry relation, defined on the set of all binary bracelets of a given length.
As before, we say that a homometry class is \emph{nontrivial} if it has cardinality greater than 1.

\newpage \begin{thm}
    \label{thm:discrete}
    Let $n$ be a positive integer, and consider the set of all binary bracelets with length $n$ and weight 5.
    Every nontrivial homometry class belongs to exactly one of the following seven types~$\mathsf{A}$--$\mathsf{G}$, where $m$ is a positive integer, and $i$ and $j$ are positive integers ranging over the index set $\I_m( \;\;)$:

    \normalfont

    \[
    \begin{array}{ll}
        \text{Type $\mathsf{A}$ (for $n = 2m$):} & \Big\{ \left[0, i, i+j, m + i - j, m \right], \;  \left[0, m+i, i+j, m + i - j, m \right] \Big\}, \\[2ex]
        & \I_m({\mathsf{A}}) = \left\{(i,j) : 1 \leq i,j < \tfrac{m}{2} \text{ and } i \neq j, \text{ and if } j = \tfrac{m}{3} \text{ then } i = \tfrac{m}{6} \right\}. \\[3ex]

        \text{Type $\mathsf{B}$ (for $n = 5m$):} & \Big\{ \left[0, i, m, 2m, 2m+i \right], \;  \left[0, i, m, m+i, 3m \right] \Big\},\\[2ex]
        & \I_m({\mathsf{B}}) = \left\{i : 1 \leq i \leq \frac{5m}{2} \right\} \setminus \left\{ \frac{m}{2}, m, \frac{3m}{2}, 2m \right\}. \\[3ex]

        \text{Type $\mathsf{C}$ (for $n = 6m$):} & \Big\{ \left[0, m + i, 2m, 2m + i, 4m \right], \;  \left[0, 2m, 2m+i, 3m + i, 4m \right] \Big\}, \\[2ex]
        & \I_m(\mathsf{C}) = \left\{ i : 1 \leq i \leq \frac{m}{2} \right\}. \\[3ex]

        \text{Type $\mathsf{D}$ (for $n = 6m$):} & \Big\{ \left[i, m - i, 2m, 2m + i, 3m \right], \;  \left[0, i, m - i, 2m + i, 5m \right] \Big\}, \\[2ex]
        & \I_m(\mathsf{D}) = \left\{ i : 1 \leq i < \frac{3m}{2} \right\} \setminus \left\{ \frac{m}{2}, m \right\}. \\[2ex]

        \text{Type $\mathsf{E}$ (for $n = 6m$):} & \Big\{ \left[0, i, m+i, 2m+i, 3m \right], \;  \left[0, i, m, 2m, 3m+i \right], \; \left[0, m-i, m, 2m, 4m-i \right] \Big\}, \\[2ex]
        & \I_m(\mathsf{E}) = \left\{ i: 1 \leq i < \frac{m}{2} \right\}. \\[3ex]

        \text{Type $\mathsf{F}$ (for $n = 8m$):} & \Big\{ \left[0, i, m, 2m + i, 4m \right], \;  \left[0, i, m + i, 2m, 4m+i \right] \Big\}, \\[2ex]
        & \I_m(\mathsf{F}) = \left\{ i : 1 \leq i < 4m \right\} \setminus \left\{ m, 2m, 3m \right\}. \\[3ex]

        \text{Type $\mathsf{G}$ (for $n = 20m$):} & \Big\{ \left[0, 5m, 8mi, 15m + 4mi, 5m - 4mi \right], \; \left[0, 5m, 4mi, 15m + 8mi, 15m - 4mi \right] \Big\}, \\[2ex]
        & \I_m(\mathsf{G}) = \left\{ i : 1 \leq i \leq 4 \right\}.
    \end{array}
    \]
\end{thm}

\begin{rem}
    Earlier in Figure~\ref{fig:k5 classification} we depicted the diagrams corresponding to each type in the classification.
    The parameters $i$ and $j$ in the diagrams coincide with those given in the theorem above; 
    in each diagram, we draw dashes dividing the circle into equal arcs, with each arc representing $m$ beads in a bracelet.
\end{rem}

\begin{proof}[Proof of Theorem~\ref{thm:discrete}]
    Consider the map
    \begin{align}
\label{bijection cts to discrete}
    \begin{split}
    \left\{ \begin{array}{c}
            \textup{5-point bracelets} \\
            \textup{with canonical vector $\mathbf{x}$} \\
            \textup{such that $n \mathbf{x} \in \Z^4$}              
        \end{array} \right\} & \longrightarrow \left\{ \begin{array}{c}
            \textup{binary bracelets} \\
            \textup{with length $n$} \\
            \textup{and weight $5$}
            \end{array} \right\}, \\[1ex]
            [0, x_1, x_2, x_3, x_4] & \longmapsto [0, nx_1, nx_2, nx_3, nx_4].
            \end{split}
\end{align}
    The map is well-defined, since both the ${\rm O}(2)$-action on $\mathbb{T}$ and the $D_{2n}$-action on $\Z_n$ are given by rotations and reflections, which commute with the scalar multiplication by $n$.
    In fact,~\eqref{bijection cts to discrete} is a bijection, since its inverse is given by taking any representative of a binary bracelet and dividing its elements by $n$.
    Comparing~\eqref{distance T} and~\eqref{distance Zn}, we observe that the distance $d_{\Z_n}$ is obtained from the distance $d_{\mathbb{T}}$ via multiplication by $n$;
    consequently, the bijection~\eqref{bijection cts to discrete} preserves the homometry relation.
    Therefore,~\eqref{bijection cts to discrete} descends to a bijection between the homometry classes in its domain and codomain.
    
    Thus in order to write down the nontrivial homometry classes in the codomain of~\eqref{bijection cts to discrete}, we simply need to determine those in Proposition~\ref{prop:classification cts} whose canonical vectors $\mathbf{x}$ are such that $n \mathbf{x}$ has integer coordinates.
    This is especially straightforward because in each homometry class given in Proposition~\ref{prop:classification cts}, there is a 5-point bracelet of the form $[0, z_2, z_3, z_4, z_5]$ given by a representative containing 0.
    Therefore, its canonical vector $\mathbf{x}$ is such that $n \mathbf{x} \in \Z^4$ if and only if each $nz_i \in \Z$.
    Since all $z_i$'s are rational, this occurs if and only if 
        \begin{itemize}
            \item $n = \ell m$ for some positive integer $m$, where $\ell$ is the least common denominator of the $z_i$'s (upon setting $t=u=0$); and
            \item $nt$ and $nu$ are integers.
        \end{itemize}
    By inspecting the 5-point bracelet representatives in Proposition~\ref{prop:classification cts}, we immediately determine the least common denominator $\ell$ as follows:
    \begin{equation}
        \label{l}
    \begin{array}{|c|c|c|c|c|c|c|c|}
    \hline
       \text{Type} & \mathsf{A} & \mathsf{B} & \mathsf{C} & \mathsf{D} & \mathsf{E} & \mathsf{F} & \mathsf{G} \\ \hline
        \ell & 2 & 5 & 6 & 6 & 6 & 8 & 20 \\ \hline
    \end{array}
    \end{equation}
    Thus in order to translate from Proposition~\ref{prop:classification cts} to Theorem~\ref{thm:discrete}, we take each 5-point bracelet representative, multiply its elements by $n$, and then express them in terms of the positive integers $m$, $i$, and $j$ using the substitutions
    \[
        m = \frac{n}{\ell}, \qquad i = nt, \qquad j = nu.
    \]
    The index sets $\I_m(\;\;)$ are obtained in the same way from the $t$- and $u$-values given in Proposition~\ref{prop:classification cts}.
    Note that for Type~$\mathsf{G}$, we use the discrete parameter $t \in \{1,2,3,4\}$ given in~\eqref{Type G parameter}.
\end{proof}

We turn to the second part of Problem~\ref{prob:main}, namely that of enumeration.
The goal is to understand the sequence $(h_n)_{n=1}^\infty$ of nonnegative integers, where
\[
    h_n \coloneqq \# \left\{ \begin{array}{c}
    \text{nontrivial homometry classes of binary bracelets} \\
    \text{with length $n$ and weight 5}
    \end{array}
\right\}.
\]
A fundamental goal in enumerative combinatorics, when studying an unknown sequence, is to find a closed form for its generating function.
In particular, the \emph{generating function} of the sequence $(h_n)$ is the formal power series
\[
H(x) \coloneqq \sum_{n=1}^\infty h_n \, x^n
\]
whose coefficients give the terms of the sequence.
If we can find a closed form (such as a rational expression) for $H(x)$, then we automatically know the entire sequence $(h_n)$ via expanding the Maclaurin series of $H(x)$, ignoring any issues of convergence.

In the following lemma, we determine the generating functions for the sequences obtained by restricting the count $h_n$ to each of the seven types in our classification.
This amounts to describing the cardinalities of the index sets $\I_m(\;\;)$ in Theorem~\ref{thm:discrete}.
In the proof, we will make use of the following well-known closed forms:
\begin{alignat}{2}
   \sum_{m=1}^\infty y^m &= \frac{y}{(1-y)},  \qquad \qquad \qquad & \sum_{m=1}^\infty \lfloor m/2 \rfloor y^m &= \frac{y^2}{(1-y)(1-y^2)}, \nonumber \\[1ex]
   \sum_{m=1}^\infty m \, y^m & = \frac{y}{(1-y)^2},    & \sum_{m=1}^\infty \lfloor (m-1)/2 \rfloor y^m &= \frac{y^3}{(1-y)(1-y^2)}, \label{gf IDs} \\[1ex]
   \sum_{m=1}^\infty (m-1) y^m &=  \frac{y^2}{(1-y)^2}, & \sum_{m=1}^\infty \lfloor (m-1)/2 \rfloor^2 \, y^m &= \frac{y^3(1+y^2)}{(1-y) (1-y^2)^2}. \nonumber
\end{alignat}

\newpage 

\begin{lemma}
\label{lemma: gfs individual}
    For each type $\bullet \in \{\mathsf{A}, \ldots, \mathsf{G}\}$ in Theorem~\ref{thm:discrete}, let
    \begin{align*}
    h_n(\bullet) &\coloneqq \# \left\{ \begin{array}{c}
    \textup{nontrivial homometry classes of binary bracelets} \\
    \textup{with length $n$ and weight 5, of Type $\bullet$}
    \end{array}  \right\}, \\[1ex]
    H(\bullet; x) & \coloneqq \sum_{n=1}^\infty h_n(\bullet) \, x^n.
    \end{align*}
    We have
    \begingroup
    \arraycolsep=5ex
    \[
    \begin{array}{ll}
        H(\mathsf{A}; x)  = \dfrac{2x^{10}}{(1-x^2)(1-x^4)^2} - \dfrac{3x^{18}}{(1-x^6)(1-x^{12})}, & H(\mathsf{E};x) = \dfrac{x^{18}}{(1-x^6)(1-x^{12})}, \\[4ex]
        H(\mathsf{B}; x) = \dfrac{x^{10} + 4x^{15}}{(1-x^5)(1-x^{10})}, & H(\mathsf{F}; x) = \dfrac{4x^{16}}{(1-x^8)^2},\\[4ex]
        H(\mathsf{C}; x) = \dfrac{x^{12}}{(1-x^6)(1-x^{12})}, & H(\mathsf{G}; x) = \dfrac{4x^{20}}{1-x^{20}}. \\[4ex]
        H(\mathsf{D}; x) =  \dfrac{3x^{18}}{(1-x^6)(1-x^{12})}, & 
        \end{array}
    \]
    \endgroup
    
\end{lemma}

\begin{proof}

    Let $\ell$ be the least common denominator given in~\eqref{l} for Type $\bullet$.
    By Theorem~\ref{thm:discrete}, we have $h_{\ell m}(\bullet) = |\I_m(\bullet)|$ for every positive integer $m$, and we have $h_n(\bullet) = 0$ if $n$ is not divisible by $\ell$.
    Therefore
    \begin{equation}
        \label{x to y}
        H(\bullet; x) = \sum_{n=1}^\infty h_n(\bullet) \, x^n  = \sum_{m=1}^\infty h_{\ell m}(\bullet) \, x^{\ell m} = \sum_{m=1}^\infty |\I_m(\bullet)| \, y^m,
    \end{equation}
    where we set $y \coloneqq x^\ell$ on the right-hand side.
    It therefore suffices to find the generating function of the sequence $(|\I_m(\bullet)|)_{m=1}^\infty$, from which we will recover $H(\bullet; x)$ by replacing $y$ with $x^\ell$.
    In each proof below, we adopt standard interval notation restricted to the integers.

    Type~$\mathsf{A}$ is the most complicated case.
    Following Theorem~\ref{thm:discrete}, the index set $\I_m(\mathsf{A})$ can be constructed in three stages.
    First take the ordered pairs $(i,j)$ in $[1,m/2) \times [1, m/2)$, of which there are $\lfloor (m-1)/2 \rfloor^2$ many; remove the diagonal pairs $(i,i)$, of which there are $\lfloor (m-1)/2 \rfloor$ many.
    Thus after this first stage, using the identities in~\eqref{gf IDs}, the generating function for the number of elements is given by
    \begin{align}
      & \phantom{=} \sum_{m=1}^\infty \lfloor (m-1)/2 \rfloor^2 \, y^m - \sum_{m=1}^\infty \lfloor (m-1)/2 \rfloor y^m \nonumber \\
      &= \frac{y^3(1+y^2)}{(1-y)(1-y^2)^2} - \frac{y^3}{(1-y)(1-y^2)} \nonumber \\ 
      &= \frac{2y^5}{(1-y)(1-y^2)^2}. \label{A first part}
    \end{align}
    As our second stage in constructing $\I_m(\mathsf{A})$, we remove all pairs $(i, m/3)$, of which there are $\lfloor (m-1)/2 \rfloor - 1$ many if $m = 3p$ for some positive integer $p$, and zero otherwise; the generating function for this number is
    \begin{align}
        \sum_{p = 1}^\infty \left( \lfloor (3p - 1)/2 \rfloor - 1 \right) y^{3p} & = \sum_{p=1}^\infty (p-1) (y^3)^p + \sum_{p=1}^\infty \lfloor (p-1)/2 \rfloor (y^3)^p \nonumber \\
        &= \frac{(y^3)^2}{(1-y^3)^2} + \frac{(y^3)^3}{(1-y^3)(1-(y^3)^2)} \nonumber \\
        &= \frac{y^6 + 2y^9}{(1-y^3)(1-y^6)}. \label{A second part}
    \end{align}
    As the third stage in constructing $\I_m(\mathsf{A})$, add back the pair $(m/6, m/3)$, of which there is 1 if $m = 6q$ for some positive integer $q$, and zero otherwise.
    This number has the generating function
    \begin{equation}
        \label{A third part}
        \sum_{q=1}^\infty y^{6q} = \frac{y^6}{1 - y^6}.
    \end{equation}
    Combining these three stages, we have
    \[
    \sum_{m=1}^\infty |\I_m(\mathsf{A})| y^m = \eqref{A first part} - \eqref{A second part} + \eqref{A third part} = \frac{2y^5}{(1-y)(1-y^2)^2} - \frac{3y^9}{(1-y^3)(1-y^6)}.
    \]
    The result for $H(\mathsf{A};x)$ follows from~\eqref{x to y}, upon setting $y = x^\ell = x^2$.
    
    Since $\I_m(\mathsf{B}) = [1, 5m/2] \setminus \{m/2, m, 3m/2, 2m\}$, we have
    $|\I_m(\mathsf{B})| = \lfloor 5m/2 \rfloor - 2 - 2 [\text{$m$ is even}]$, where the final bracket is the Iverson bracket (1 if its argument is true, and 0 otherwise).
    Thus, 
    \begin{align*}
        \sum_{m=1}^\infty |\I_m(\mathsf{B})| \, y^m &= \sum_{m=1}^\infty (\lfloor 5m/2 \rfloor - 2 ) y^m - 2\sum_{p = 1}^\infty y^{2p} \\
        &= \sum_{m=1}^\infty 2m y^m + \sum_{m=1}^\infty \lfloor m/2 \rfloor y^m - 2\sum_{m=1}^\infty y^m - 2 \sum_{p=1}^\infty (y^2)^p\\
        &= \frac{2y}{(1-y)^2} + \frac{y^2}{(1-y)(1-y^2)} - \frac{2y}{1-y} - \frac{2y^2}{1-y^2} \\
        &= \frac{y^2 + 4y^3}{(1-y)(1-y^2)},
    \end{align*}
and the result for $H(\mathsf{B};x)$ follows from~\eqref{x to y}, upon setting $y = x^\ell = x^5$.

Since $\I_m(\mathsf{C}) = [1, m/2]$, we have $|\I_m(\mathsf{C})| = \lfloor m/2 \rfloor$.
Therefore 
\[
\sum_{m=1}^\infty |\I_m(\mathsf{C})| y^m = \sum_{m=1}^\infty \lfloor m/2 \rfloor y^m = \frac{y^2}{(1-y)(1-y^2)},
\]
and the result for $H(\mathsf{C};x)$ follows from~\eqref{x to y}, upon setting $y = x^\ell = x^6$.

Since $\I_m(\mathsf{D}) = [1, 3m/2) \setminus \{m/2, m\}$, we have 
$|\I_m(\mathsf{D})| = \lfloor (3m-1)/2 \rfloor - 1 - [\text{$m$ is even}]$.
Thus
\begin{align*}
    \sum_{m=1}^\infty |\I_m(\mathsf{D})| y^m & = \sum_{m=1}^\infty (\lfloor (3m-1)/2 \rfloor - 1) y^m  - \sum_{p=1}^\infty y^{2p}\\
    &= \sum_{m=1}^\infty (m-1) \, y^m + \sum_{m=1}^\infty \lfloor (m-1)/2 \rfloor y^m - \sum_{p=1}^\infty (y^2)^p \\
    &= \frac{y^2}{(1-y)^2} + \frac{y^3}{(1-y)(1-y^2)} - \frac{y^2}{1-y^2} \\
    &= \frac{3y^3}{(1-y)(1-y^2)},
\end{align*}
and the result for $H(\mathsf{D};x)$ follows from~\eqref{x to y}, upon setting $y = x^\ell = x^6$.

Since $\I_m(\mathsf{E}) = [1, m/2)$, we have $|\I_m(\mathsf{E})| = \lfloor (m-1)/2 \rfloor$.
Thus
\[
\sum_{m=1}^\infty |\I_m(\mathsf{E})| y^m = \sum_{m=1}^\infty \lfloor (m-1)/2 \rfloor = \frac{y^3}{(1-y)(1-y^2)},
\]
and the result for $H(\mathsf{E};x)$ follows from~\eqref{x to y}, upon setting $y = x^\ell = x^6$.

Since $\I_m(\mathsf{F}) = [1, 4m) \setminus \{m, 2m, 3m\}$, we have $|\I_m(\mathsf{F})| = 4m - 4 = 4(m-1)$.
Thus
\[
\sum_{m=1}^\infty |\I_m(\mathsf{F})| y^m = 4 \sum_{m=1}^\infty (m-1) y^m = \frac{4y^2}{(1-y)^2},
\]
and the result for $H(\mathsf{F};x)$ follows from~\eqref{x to y}, upon setting $y = x^\ell = x^8$.

Since $\I_m(\mathsf{G}) = [1,4]$, we have $|\I_m(\mathsf{G})| = 4$ for all $m$.
Thus
\[
\sum_{m=1}^\infty |\I_m(\mathsf{G})| y^m = 4 \sum_{m=1}^\infty y^m = \frac{4y}{1-y},
\]
and the result for $H(\mathsf{G};x)$ follows from~\eqref{x to y}, upon setting $y = x^\ell = x^{20}$.
\end{proof}

\begin{thm}
    \label{thm:gf}
    Let $h_n$ denote the number of nontrivial homometry classes of binary bracelets with length $n$ and weight 5.
    The generating function $H(x)$ of the sequence $(h_n)_{n=1}^\infty$ has the closed form
    \begin{align*}
        H(x) & \coloneqq \sum_{n=1}^\infty h_n \, x^n \\
        &= \frac{2x^{10}}{(1-x^2)(1-x^4)^2} + \frac{x^{10}+4x^{15}}{(1-x^5)(1-x^{10})} + \frac{x^{12} + x^{18}}{(1-x^6)(1-x^{12})} + \frac{4x^{16}}{(1-x^8)^2} + \frac{4x^{20}}{1-x^{20}}.
    \end{align*}
    
\end{thm}

\begin{proof}

By Theorem~\ref{thm:discrete}, Types~$\mathsf{A}$--$\mathsf{G}$ are exhaustive and pairwise disjoint.
Therefore we simply add the seven type-specific generating functions to obtain $H(x) = H(\mathsf{A}; x) + \cdots + H(\mathsf{G}; x)$.
The result follows from Lemma~\ref{lemma: gfs individual}, upon some light simplification.    
\end{proof}

Inspecting the numerators and denominators in the rational functions in Theorem~\ref{thm:gf}, we see that
\[
h_n > 0 \text{ if and only if $n \geq 10$ and $n$ is divisible by $2$ or $5$}. 
\]
Concretely, starting at $n=10$, the generating function $H(x)$ yields the following initial values of $h_n$:
\[
(h_n)_{n=10}^\infty = (3, 0, 3, 0, 6, 5, 10, 0, 14, 0, 22, 0, 20, 0, 31, 10, 30, 0, 30, 0, 57, 0, 54, 0, 56, 15, 61, 0, 72, 0, 108, \ldots).
\]
Note that by replacing the term $x^{18}$ with $tx^{18}$ in Theorem~\ref{thm:gf}, we can refine $H(x)$ with respect to the size of the homometry classes.
That is, the resulting bivariate power series reveals the number $h_n^{(2)}$ of homometric pairs and the number $h_n^{(3)}$ of homometric triples, since its expansion is
\begin{equation}
    \label{gf refined}
    \sum_{n=1}^\infty \left( h_n^{(2)} + h_n^{(3)}t \right) x^n.
\end{equation}
This follows from the fact that Type~$\mathsf{E}$ is the only type consisting of homometric triples rather than pairs;
indeed, changing $x^{18}$ to $t x^{18}$ is the same as changing $H(\mathsf{E};x)$ to $t \cdot H(\mathsf{E};x)$ in the proof of Theorem~\ref{thm:discrete}.

\begin{rem}
    The sum of rational functions in Theorem~\ref{thm:gf} is especially amenable to extracting asymptotic information about the values $h_n$.
    In particular, the five rational functions have poles at $x=1$ of orders 3, 2, 2, 2, and 1, respectively;
    therefore the pole of order 3 (coming from the first rational function) dominates, and consequently $h_n = O(n^{3-1}) = O(n^2)$.
    To pin down the constant, we again look at the first rational function in Theorem~\ref{thm:gf},
    and take the limit of its product with $(1-x)^3$ as $x \rightarrow 1$:
    \[
    \lim_{x \rightarrow 1} \left[ (1-x)^3 \cdot \frac{2x^{10}}{(1-x^2)(1-x^4)^2} \right] = \frac{1}{16}.
    \]
    Since the first rational function is a function of $x^2$, it contributes nothing to $h_n$ for $n$ odd;
    hence we have the asymptotic
    \begin{equation}
        \label{asymptotic}
        h_n \sim \frac{n^2}{16} \text{ as $n \rightarrow \infty$ (for $n$ even)},
    \end{equation}
    while $h_n = O(n)$ for $n$ odd.
    On the left-hand side below, to illustrate~\eqref{asymptotic}, we plot the values of $h_n$ for $n = 1, \ldots, 200$, along with the red parabola $n^2/16$.
    On the right-hand side below, we do the same thing for $n = 20,40,60, \ldots, 1000$.

    \begin{center} 
    \includegraphics[height=3cm]{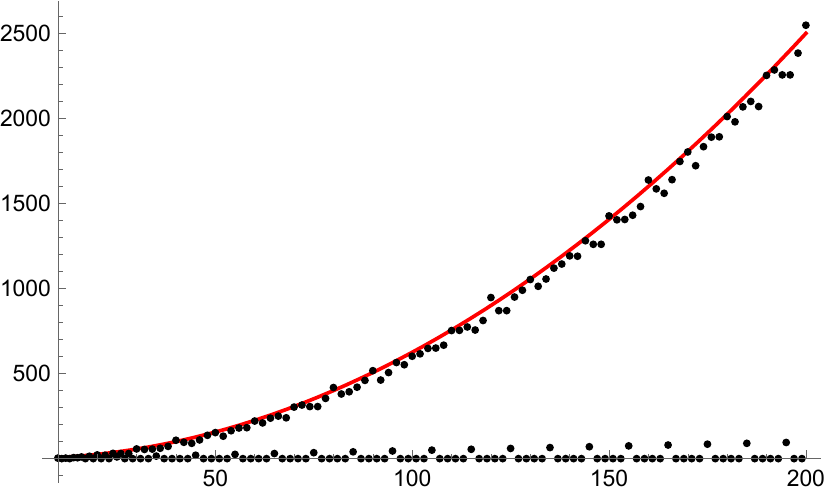}
    \hspace{1cm}
    \includegraphics[height=3cm]{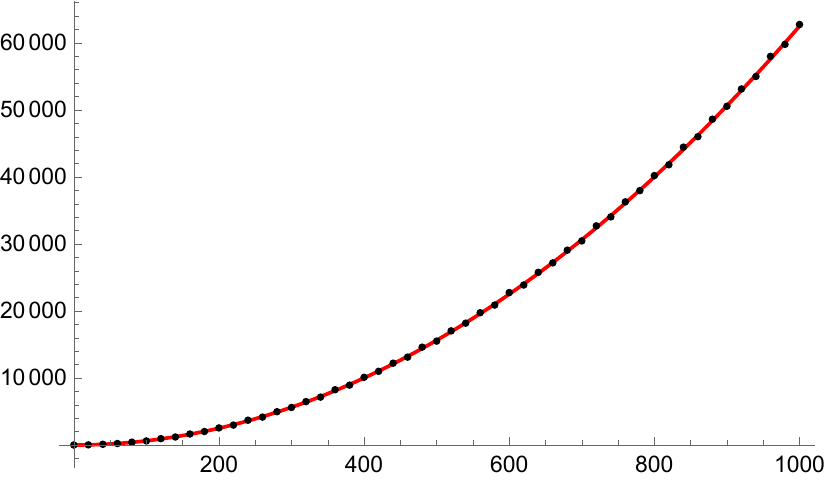}
    \end{center}
\end{rem}

\section{Open problems}

\label{sec:open}

The results in this paper suggest several avenues for further research.

\begin{prob}
\label{prob:criterion}
    Give a combinatorial characterization of the 22 difference tables in $\bigcup_{(p,q)} \M(p,q)$.
\end{prob}

Essentially, Problem~\ref{prob:criterion} aims to obviate the need for Algorithm~\ref{alg:Mpq}, by finding an \emph{a priori} condition that cuts out a minimal set $\M(p,q)$.
As a first step, one can easily write down certain simple sufficient conditions for $\XY(\Delta)$ to be empty; for example, due to the relations~\eqref{sums ID}, for $h \leq i < j \leq k$, we have $\XY(\Delta) = \varnothing$ if the entry $\boxed{\mathsf{hk}}$ lies weakly southwest of the entry $\boxed{\mathsf{ij}}$.
One can write down similar conditions corresponding to the various inequalities in~\eqref{positive}--\eqref{noncongruent}, but there are certainly other less obvious configurations that produce an empty solution set.
Ultimately, if $\Delta$ and $\Delta'$ arise from the permutations $\pi$ and $\pi'$, respectively, then one might hope to find a certain condition on $\pi$ and $\pi'$ that is equivalent to $\XY(\Delta) \subseteq \XY(\Delta')$.
Recall that the minimal set of 22 tables which we give in this paper is unique up to substituting tables with identical solution sets; hence it is conceivable that a different choice of representative tables could reveal a combinatorial pattern that is currently hidden.

\begin{prob}
    Extend the methods in this paper to solve Problem~\ref{prob:main} for weights $k\geq 6$.
\end{prob}

On one hand, theoretically, our methods depended very little on the fact that $k=5$.
On the other hand, the jump from $k=5$ to $k=6$ presents several practical complications.
First, for $k=6$ the number of pairwise distances is $\binom{6}{2} = 15$, meaning that the number of permutations of $D$ is now $15! \approx 1.3 \times 10^{12}$, so that an analogue of Algorithm~\ref{alg:Mpq} is far less feasible.
(Thus, any hope of solving the $k=6$ case seems to depend on first solving Problem~\ref{prob:criterion} above.)
Moreover, for $k=6$ we would expect (possibly many) more than just five valid pairs $(p,q)$, given the appropriate analogue of ``long count'' for a cyclic hexagon.
Add to this the fact, mentioned in the introduction, that unlike the $k \leq 5$ case, for $k \geq 6$ there are homometric subsets of $\Z$ with cardinality $k$; since all of these descend to homometric bracelets as well, we would expect this fact to complicate the attempt at solving Problem~\ref{prob:main}.

Our original draft posed the following problem, which has since been solved by P.~Methawisal~\cite{Methawisal}: \emph{Obtain a group-theoretic version of the classification in Theorem~\ref{thm:discrete}.}
Naturally we refer the reader to~\cite{Methawisal} for details, but to motivate this problem briefly, we point out the following ``hidden symmetries'' among homometric bracelets.
Let $U_n \coloneqq \Z_n^\times$ denote the group of multiplicative units in the ring $\Z_n$, and let
\[
\widetilde{U}_n \coloneqq U_n / \{\pm 1\}.
\]
Then $\widetilde{U}_n$ acts on a bracelet of length $n$ by multiplication (modulo $n$) on any of the bracelet representatives:
\[
u \cdot [s_1, \ldots, s_k] = [us_1, \ldots, us_k].
\]
One can easily show that this gives a well-defined group action of $\widetilde{U}_n$ on the set of all length-$n$ binary bracelets with some fixed weight $k$.
(This is of course true of $U_n$ as well, but since bracelets are unchanged under reflection, one may as well work modulo sign to obtain a faithful action.)
Moreover, this action preserves the homometry relation, and so it actually extends to a $\widetilde{U}_n$-action on the set of all homometry classes of such bracelets.
Relating to this paper, then, a natural problem was to describe the structure of the $\widetilde{U}_n$-orbits in the set of nontrivial homometry classes, including a characterization of the fixed points (i.e., nontrivial homometry classes whose bracelets are maximally ``symmetrical'' with respect to the $\widetilde{U}_n$-action).
Moreover, we conjectured the following refinement, which now has been proved in the same paper~\cite{Methawisal} cited above:

\begin{thm}[Methawisal~\cite{Methawisal}*{Thm.~3.1}]
    With respect to the action given above, the group $\widetilde{U}_n$ acts on the set of nontrivial homometry classes of each Type $\mathsf{A}$--$\mathsf{G}$ in Theorem~\ref{thm:discrete}.
\end{thm}

We again refer the reader to~\cite{Methawisal} for a detailed analysis of the $\widetilde{U}_n$-action on each of our seven types.

\appendix

\section{Full example: the difference table $\mathsf{D_4}$ and its solution set}
\label{app:example}

In this example, we illustrate Definitions~\ref{def:difference table} and~\ref{def:X and XY} by constructing a difference table from a permutation $\pi$, and then explicitly writing down the system defining its solution set.
(It turns out that this difference table is the element in $\M(1,2)$ that we labeled as $\mathsf{D_4}$.)
The example is meant to make Algorithm~\ref{alg:Mpq} completely transparent and easy to program.

First we follow Definition~\ref{def:difference table}, where $(p,q) = (1,2)$.
Let $\pi$ be the permutation of $D$ that acts on~\eqref{table X} as shown in the first step below:

\begin{equation}
\label{example three steps}
\begin{ytableau}[\mathsf]
    {01} & {02} & {03} & {04} \\
    \none & {12} & {13} & {14} \\
    \none & \none & {23} & {24}\\
    \none & \none & \none & {34}
\end{ytableau} 
\quad
\xrightarrow{\quad\pi\quad}
\begin{ytableau}[\mathsf]
    {02} & {03} & {14} & {24} \\
    \none & {23} & {34} & {04} \\
    \none & \none & {13} & {01}\\
    \none & \none & \none & {12}
\end{ytableau}
\quad
\xrightarrow{\substack{\text{bars on} \\ \text{$p$-long} \\ \text{entries}}}
\begin{ytableau}[\mathsf]
    {02} & {03} & {14} & {24} \\
    \none & {23} & {34} & \mybar{04} \\
    \none & \none & {13} & {01}\\
    \none & \none & \none & {12}
\end{ytableau}
\quad
\xrightarrow{\substack{\text{bars on} \\ \text{$q$-long} \\ \text{positions}}}
\begin{ytableau}[\mathsf]
    {02} & {03} & \mybar{14} & \mybar{24} \\
    \none & {23} & {34} & \mybar{04} \\
    \none & \none & {13} & {01}\\
    \none & \none & \none & {12}
\end{ytableau}
\end{equation}

\noindent Since $p=1$, there is one $p$-long element of $D$, namely $\boxed{\mathsf{04}}$, by~\eqref{long in X}.
Thus in the second step in~\eqref{example three steps}, we place a bar over the entry $\boxed{\mathsf{04}}$ in the permuted table.
Finally, since $q=2$, there are two $q$-long elements of $D$, namely $\boxed{\mathsf{03}}$ and $\boxed{\mathsf{04}}$, by~\eqref{long in X}.
Referring to the original (unpermuted) table on the far left of~\eqref{example three steps}, we observe the positions of these two entries, namely the two rightmost boxes in the top row.
Thus in the third step in~\eqref{example three steps}, we place bars over the entries in the corresponding positions in the permuted table, namely $\boxed{\mathsf{14}}$ and $\boxed{\mathsf{24}}$.
(If one of these positions had already had a bar, then we would have removed the bar, since two bars cancel each other out.)

The rightmost table in~\eqref{example three steps} is the element $\Delta \in \D(1,2)$ corresponding to the permutation $\pi$. 
We should view this $(1,2)$-difference table $\Delta$ as a generic pair of homometric 5-point bracelets $[X]$ and $[Y]$ with long counts $1$ and $2$, respectively:

\begin{center} 
\tikzset{
  midlabel/.style={
    draw=black,
    fill=white,
    text=black,
    inner sep=2pt,
    rectangle,
    font=\tiny
  }
}
\begin{tikzpicture}[scale=2, baseline]

    \node[anchor=base] at (0,0) {};
    
    \draw[thick] (0,0) circle (1);

    \node (a) at (90:1) [dot] {};
    \node (b) at (50:1) [dot] {};
    \node (c) at (5:1) [dot] {};
    \node (d) at (-40:1) [dot] {};
    \node (e) at (-110:1) [dot] {};

    \node[draw,circle, fill=black, inner sep=0pt, minimum width=2pt] at (0,0) {};

    \draw[thick, lightgray] (a) -- (b) node[midlabel, midway] {$\mathsf{01}$} -- (c) node[midlabel, midway] {$\mathsf{12}$} -- (d) node[midlabel, midway] {$\mathsf{23}$} -- (e) node[midlabel, midway] {$\mathsf{34}$} --(a) node[midlabel, pos=.6] {$\mathsf{\overline{04}}$} (a) -- (c) node[midlabel, pos=.35] {$\mathsf{02}$} (a) -- (d) node[midlabel, pos=.32] {$\mathsf{03}$} (b) -- (d) node[midlabel, midway] {$\mathsf{13}$} (b) -- (e) node[midlabel, pos=.7] {$\mathsf{14}$} (c) -- (e) node[midlabel, midway] {$\mathsf{24}$};

    \node at (0,-1.2) {$[X]$};
    
\end{tikzpicture}
\hspace{10ex}
\begin{tikzpicture}[scale=2, baseline]
    \node[anchor=base] at (0,0) {};
    
    \draw[thick] (0,0) circle (1);

    \node (a) at (90:1) [dot] {};
    \node (b) at (0:1) [dot] {};
    \node (c) at (-50:1) [dot] {};
    \node (d) at (-110:1) [dot] {};
    \node (e) at (-150:1) [dot] {};

    \node[draw,circle, fill=black, inner sep=0pt, minimum width=2pt] at (0,0) {};

    \draw[thick, lightgray] (a) -- (b) node[midlabel, midway] {$\mathsf{02}$} -- (c) node[midlabel, midway] {$\mathsf{23}$} -- (d) node[midlabel, midway] {$\mathsf{13}$} -- (e) node[midlabel, midway] {$\mathsf{12}$} --(a) node[midlabel, midway] {$\mathsf{24}$} (a) -- (c) node[midlabel, pos=.35] {$\mathsf{03}$} (a) -- (d) node[midlabel, pos=.32] {$\mathsf{14}$} (b) -- (d) node[midlabel, midway] {$\mathsf{34}$} (b) -- (e) node[midlabel, midway] {$\mathsf{\overline{04}}$} (c) -- (e) node[midlabel, midway] {$\mathsf{01}$};

    \node at (0,-1.2) {$[Y]$};
    
\end{tikzpicture}
\end{center}

\noindent In the diagrams above, we label each chord with the \emph{distance} (not the difference!) between its endpoints.
Note that these labels are precisely the entries obtained after the \emph{second} step in~\eqref{example three steps}, since that step renders all entries short.

Next we will construct the solution set $\mathbf{XY}\!_{1,2}(\Delta)$, following Definition~\ref{def:X and XY}.
Expanding the shorthand in~\eqref{positive}--\eqref{ineqs canonical ordering p}, respectively, where $p=1$, we obtain the following three constraints:
\begin{align}
\label{p constraints example}
\begin{split}
    &0 < x_1 < x_2 < x_3 < x_4,\\
    &x_4 > 1/2 \text{ and } x_3 \leq 1/2 \text{ and } x_4 - x_1 \leq 1/2,\\
    &x_1 \leq x_4 - x_3, \text{ and if }x_1 = x_4 - x_3 \text{ then } x_2 - x_1 \leq x_3 - x_2.  
\end{split}
\end{align}
To write out the six equations~\eqref{eqs six sums}, we simply refer to the table $\Delta$ on the far right of~\eqref{example three steps}, and force each off-diagonal entry to be the sum of the diagonal entries lying weakly southwest of it:
\begin{equation}
\label{p equations example}
\begin{array}{ccc}
\begin{aligned}
    \boxed{\mathsf{03}} & = \boxed{\mathsf{02}} + \boxed{\mathsf{23}}, \\
    \boxed{\mathsf{34}} & = \boxed{\mathsf{23}} + \boxed{\mathsf{13}}, \\
    \boxed{\mathsf{01}} & = \boxed{\mathsf{13}} + \boxed{\mathsf{12}}, \\
    \boxed{\mathsf{\overline{14}}} &= \boxed{\mathsf{02}} + \boxed{\mathsf{23}} + \boxed{\mathsf{13}}, \\
    \boxed{\mathsf{\overline{04}}} & = \boxed{\mathsf{23}} + \boxed{\mathsf{13}} + \boxed{\mathsf{12}}, \\
    \boxed{\mathsf{\overline{24}}} & = \boxed{\mathsf{02}} + \boxed{\mathsf{23}} + \boxed{\mathsf{13}} + \boxed{\mathsf{12}}.
\end{aligned}
& \leadsto \quad & \begin{bmatrix}
    \phantom{-}0 & \phantom{-}0 & 0 & \phantom{-}0 \\
    -1 & -1 & 2 & -1\\
    -3 & \phantom{-}1 & 1 & \phantom{-}0\\
    -2 & \phantom{-}0 & 2 & \phantom{-}1\\
    -2 & \phantom{-}0 & 2 & \phantom{-}1\\
    -2 & \phantom{-}0 & 2 & \phantom{-}1
\end{bmatrix} \begin{bmatrix}
    x_1 \\ x_2 \\ x_3 \\ x_4
\end{bmatrix} = \begin{bmatrix}
    0 \\ 0 \\ 0 \\ 1 \\ 1 \\ 1
\end{bmatrix}.
\end{array}
\end{equation}
To write out~\eqref{same but for q}, we do the same as we did for~\eqref{p constraints example} above, where $q=2$ and where we take the entries from the positions in $\Delta$ corresponding to the original entries in the unpermuted table~\eqref{table X}:
\begin{align}
    \label{q constraints example}
    \begin{split}
    &1 - (x_4 - x_1) > 1/2 \text{ and } x_3 < 1/2 \text{ and } 1 - x_4 \leq 1/2,\\
    &x_2 \leq 1-(1-(x_4 - x_2)), \text{ and if } x_2 = 1-(1-(x_4 - x_2)) \text{ then } x_3 - x_2 \leq x_2 - x_1.
    \end{split}
\end{align}
Finally, recalling that $\mathbf{y}(\Delta)$ is given by the top row of $\Delta$, namely
\begin{equation}
\label{yDelta example}
    \mathbf{y}(\Delta) = \left( \boxed{\mathsf{02}}, \boxed{\mathsf{03}}, \boxed{\mathsf{\overline{14}}}, \boxed{\mathsf{\overline{24}}} \right) = \big(x_2, \; x_3, \; 1-(x_4-x_1), \; 1 - (x_4 - x_2) \big),
\end{equation}  
the constraint~\eqref{noncongruent} becomes
\begin{equation}
    \label{noncongruent example}
    x_1 \neq x_2 \text{ or } x_2 \neq x_3 \text{ or } x_3 \neq 1 - (x_4 - x_1) \text{ or } x_4 \neq 1 - (x_4 - x_2).
\end{equation}
Using software to solve~\eqref{p constraints example}--\eqref{noncongruent example} for $\mathbf{x} = (x_1, x_2, x_3, x_4)$, we obtain the following line segment in $\R^4$:
\[
\mathbf{X}_{1,2}(\Delta) = \left\{ \myvec{x_1 \\ 2x_1 - 1/6 \\ 1/6 + x_1 \\ 2/3} : 1/6 < x_1 \leq 1/4 \right\}.
\]
Thus by~\eqref{yDelta example}, we obtain the solution set
\[
\mathbf{XY}\!_{1,2}(\Delta) = \left\{ \left( \myvec{x_1 \\ 2x_1 - 1/6 \\ 1/6 + x_1 \\ 2/3}_{\textstyle ,} \myvec{2x_1 - 1/6 \\ 1/6 + x_1 \\ 1/3 + x_1 \\ 1/6 + 2x_1} \right) : 1/6 < x_1 \leq 1/4 \right\}. 
\]
This concludes the example.
(Compare with the $\Delta = \mathsf{D_4}$ entry in the $\M(1,2)$ table in Appendix~\ref{app:tables}.)

\newpage

\section{The solution sets $\XY(\Delta)$ before reparametrization}
\label{app:tables}

In each table below (corresponding to each of the five $(p,q)$ pairs in Lemma~\ref{lemma: valid pq}), we write down the explicit solution set $\XY(\Delta)$ for each $\Delta \in \M(p,q)$.
These results are easily obtained by programming the system~\eqref{positive}--\eqref{noncongruent} given in Definition~\ref{def:X and XY}.
Each of the 22 solution sets below appears in the proof of Proposition~\ref{prop:classification cts}, albeit in a carefully reparametrized form.
We provide the ``raw data'' here so that one can easily verify the reparametrizations given in that proof above.

\setlength{\LTleft}{0pt}

\begin{longtable}{|c|c|l|}
\hline
     & $\Delta \in \M(0,1)$ & \vpad{\mathbf{XY}\!_{0,1}(\Delta)} \\ \hline

    $\mathsf{A_1}$ & \smallpad{\begin{ytableau}[\mathsf]
    {01} & {12} & {13} & \mybar{14} \\
    \none & {34} & {24} & {04} \\
    \none & \none & {23} & {03}\\
    \none & \none & \none & {02}
\end{ytableau}} & $\left\{ \left( \myvec{x_1 \\ x_2 \\ 1/2 + 2x_1 - x_2 \\ 1/2}_{\textstyle ,} \myvec{x_1\\x_2-x_1\\ 1/2 + x_1 - x_2 \\ 1/2 + x_1} \right) : \begin{array}{c} 0 < 3x_1 < x_2 < x_1 + 1/4\\\text{or}\\0 < 3x_1 = x_2 \leq 1/4\\ \end{array} \right\}$ \\ \hline

$\mathsf{A_2}$ & \smallpad{\begin{ytableau}[\mathsf]
    {34} & {23} & {13} & \mybar{03} \\
    \none & {01} & {02} & {04} \\
    \none & \none & {12} & {14}\\
    \none & \none & \none & {24}
\end{ytableau}} & $\left\{ \left( \myvec{2x_3 - x_2 - 1/2 \\ x_2 \\ x_3 \\ 1/2}_{\textstyle ,} \myvec{1/2 - x_3 \\ x_3 - x_2 \\ 1/2 + x_2 - x_3 \\ 1 - x_3} \right) : \begin{array}{c} x_2 < 2x_3 - 1/2 < 2x_2 \text{ and } 3x_3 - 1 < x_2;\\ \text{or} \\ 1/2 < 3x_3 - 1 = x_2 \leq 1/4 \end{array} \right\}$ \\ \hline

$\mathsf{D_1}$ & \smallpad{\begin{ytableau}[\mathsf]
    {24} & {12} & {02} & \mybar{04} \\
    \none & {23} & {34} & {14} \\
    \none & \none & {01} & {03}\\
    \none & \none & \none & {13}
\end{ytableau}} & $\left\{ \left( \myvec{2x_2 - 1/2 \\ x_2 \\ 1/3 \\ 1/6 + x_2}_{\textstyle ,} \myvec{1/6 \\ 1/2 - x_2 \\ x_2 \\ 5/6 - x_2} \right) : 1/4 < x_2 < 1/3 \right\}$ \\ \hline

$\mathsf{D_2}$ & \smallpad{\begin{ytableau}[\mathsf]
    {24} & {12} & {02} & \mybar{14} \\
    \none & {34} & {23} & {04} \\
    \none & \none & {01} & {03}\\
    \none & \none & \none & {13}
\end{ytableau}} & $\left\{ \left( \myvec{2x_2 - 1/2 \\ x_2 \\ 2x_2 - 1/6 \\ 1/6 + x_2}_{\textstyle ,} \myvec{1/6 \\ 1/2 - x_2 \\ x_2 \\ 1/3 + x_2} \right) : 1/4 < x_2 < 5/18 \right\}$ \\ \hline

$\mathsf{D_3}$ & \smallpad{\begin{ytableau}[\mathsf]
    {02} & {23} & {24} & \mybar{03} \\
    \none & {01} & {12} & {04} \\
    \none & \none & {34} & {14}\\
    \none & \none & \none & {13}
\end{ytableau}} & $\left\{ \left( \myvec{x_1 \\ 1/6 \\ 1/3 + x_1 \\ 1/2 - x_1}_{\textstyle ,} \myvec{1/6 \\ 1/6 + x_1 \\ 1/3 - x_1 \\ 2/3 - x_1} \right) : 0 < x_1 \leq 1/18 \right\}$ \\ \hline

$\mathsf{F_1}$ & \smallpad{\begin{ytableau}[\mathsf]
    {01} & {23} & {13} & \mybar{14} \\
    \none & {02} & {34} & {04} \\
    \none & \none & {12} & {24}\\
    \none & \none & \none & {03}
\end{ytableau}} & $\left\{ \left( \myvec{x_1 \\ 1/8 \\ 1/4 + x_1 \\ 1/2}_{\textstyle ,} \myvec{x_1 \\ 1/8 + x_1 \\ 1/4 \\ 1/2 + x_1} \right) : 0 < x_1 < 1/8 \right\}$ \\ \hline

$\mathsf{F_2}$ & \smallpad{\begin{ytableau}[\mathsf]
    {24} & {12} & {13} & \mybar{02} \\
    \none & {01} & {34} & {04} \\
    \none & \none & {23} & {14}\\
    \none & \none & \none & {03}
\end{ytableau}} & $\left\{ \left( \myvec{x_1 \\ 1/4 + x_1 \\ 3/8 \\ 1/2}_{\textstyle ,} \myvec{1/4 - x_1 \\ 1/4 \\ 3/8 - x_1 \\ 3/4 - x_1} \right) : 0 < x_1 < 1/8 \right\}$ \\ \hline

$\mathsf{G_1}$ & \smallpad{\begin{ytableau}[\mathsf]
    {34} & {23} & {13} & \mybar{04} \\
    \none & {01} & {02} & {14} \\
    \none & \none & {12} & {24}\\
    \none & \none & \none & {03}
\end{ytableau}} & $\left\{ \left( \myvec{1/20 \\ 2/20 \\ 6/20 \\ 9/20}_{\textstyle ,} \myvec{3/20 \\ 4/20 \\ 5/20 \\ 11/20} \right) \right\}$ \\ \hline

$\mathsf{G_2}$ & \smallpad{\begin{ytableau}[\mathsf]
    {12} & {02} & {13} & \mybar{04} \\
    \none & {01} & {23} & {14} \\
    \none & \none & {34} & {03}\\
    \none & \none & \none & {24}
\end{ytableau}} & $\left\{ \left( \myvec{1/20 \\ 4/20 \\ 7/20 \\ 9/20}_{\textstyle ,} \myvec{3/20 \\ 4/20 \\ 6/20 \\ 11/20} \right) \right\}$ \\ \hline
\end{longtable}

\begin{longtable}{|c|c|l|}
\hline
     & $\Delta \in \M(0,2)$ & \vpad{\mathbf{XY}\!_{0,2}(\Delta)} \\ \hline

$\mathsf{B_1}$ & \smallpad{\begin{ytableau}[\mathsf]
    {13} & {03} & \mybar{04} & \mybar{14} \\
    \none & {01} & {02} & {24} \\
    \none & \none & {12} & {23}\\
    \none & \none & \none & {34}
\end{ytableau}} & $\left\{ \left( \myvec{x_1 \\ 1/5 \\ 2/5 \\ 2/5 + x_1}_{\textstyle ,} \myvec{2/5 - x_1 \\ 2/5 \\ 3/5 - x_1 \\ 3/5} \right) : 0 < x_1 < 1/10 \right\}$ \\ \hline
\end{longtable}

\begin{longtable}{|c|c|l|}
\hline
     & $\Delta \in \M(1,1)$ & \vpad{\mathbf{XY}\!_{1,1}(\Delta)} \\ \hline

    $\mathsf{A_3}$ & \smallpad{\begin{ytableau}[\mathsf]
    {12} & {02} & {03} & \mybar{24} \\
    \none & {01} & {34} & {14} \\
    \none & \none & {23} & \mybar{04}\\
    \none & \none & \none & {13}
\end{ytableau}} & $\left\{ \left( \myvec{x_1 \\ 2x_3 - 1/2 \\ x_3 \\ 1/2 + x_1}_{\textstyle ,} \myvec{2x_3 - x_1 - 1/2 \\ 2x_3 - 1/2 \\ x_3 \\ 2x_3 - x_1} \right) : 0 < x_1 < x_3 - 1/4 < 1/4 \right\}$ \\ \hline

$\mathsf{B_2}$ & \smallpad{\begin{ytableau}[\mathsf]
    {01} & {02} & \mybar{04} & \mybar{14} \\
    \none & {12} & {34} & {24} \\
    \none & \none & {13} & {03}\\
    \none & \none & \none & {23}
\end{ytableau}} & $\left\{ \left( \myvec{x_1 \\ 1/5 \\ 1/5 + x_1 \\ 3/5}_{\textstyle ,} \myvec{x_1 \\ 1/5 \\ 2/5 \\ 2/5 + x_1} \right) : 1/10 < x_1 < 1/5 \right\}$ \\ \hline

$\mathsf{B_3}$ & \smallpad{\begin{ytableau}[\mathsf]
    {12} & {13} & {03} & \mybar{02} \\
    \none & {23} & {34} & {14} \\
    \none & \none & {01} & \mybar{04}\\
    \none & \none & \none & {24}
\end{ytableau}} & $\left\{ \left( \myvec{x_1 \\ 1/5 + x_1 \\ 2/5 \\ 3/5}_{\textstyle ,} \myvec{1/5 \\ 2/5 - x_1 \\ 2/5 \\ 4/5 - x_1} \right) : 1/10 \leq x_1 < 1/5 \right\}$ \\ \hline

$\mathsf{B_6}$ & \smallpad{\begin{ytableau}[\mathsf]
    {02} & \mybar{04} & {03} & \mybar{12} \\
    \none & {01} & {23} & {14} \\
    \none & \none & {34} & {13}\\
    \none & \none & \none & {24}
\end{ytableau}} & $\left\{ \left( \myvec{1/10 \\ 3/10 \\ 5/10 \\ 6/10}_{\textstyle ,} \myvec{3/10 \\ 4/10 \\ 5/10 \\ 8/10} \right) \right\}$ \\ \hline

$\mathsf{F_3}$ & \smallpad{\begin{ytableau}[\mathsf]
    {34} & {23} & {03} & \mybar{13} \\
    \none & {01} & {24} & {14} \\
    \none & \none & {02} & \mybar{04}\\
    \none & \none & \none & {12}
\end{ytableau}} & $\left\{ \left( \myvec{x_1 \\ 1/4 \\ 3/8 + x_1 \\ 1/2 + x_1}_{\textstyle ,} \myvec{1/8 \\ 1/8 + x_1 \\ 3/8 + x_1 \\ 5/8} \right) : 0 < x_1 < 1/8 \right\}$ \\ \hline
\end{longtable}

\begin{longtable}{|c|c|l|}
\hline
     & $\Delta \in \M(1,2)$ & \vpad{\mathbf{XY}\!_{1,2}(\Delta)} \\ \hline

$\mathsf{C}$ & \smallpad{\begin{ytableau}[\mathsf]
    {02} & {03} & \mybar{14} & {04} \\
    \none & {23} & {01} & {24} \\
    \none & \none & {13} & {34}\\
    \none & \none & \none & {12}
\end{ytableau}} & $\left\{ \left( \myvec{x_1 \\ 1/3 \\ 1/6 + x_1 \\ 2/3}_{\textstyle ,} \myvec{1/3 \\ 1/6 + x_1 \\ 1/3 + x_1 \\ 2/3} \right) : 1/6 < x_1 \leq 1/4 \right\}$ \\ \hline

 $\mathsf{D_4}$ & \smallpad{\begin{ytableau}[\mathsf]
    {02} & {03} & \mybar{14} & \mybar{24} \\
    \none & {23} & {34} & \mybar{04} \\
    \none & \none & {13} & {01}\\
    \none & \none & \none & {12}
\end{ytableau}} & $\left\{ \left( \myvec{x_1 \\ 2x_1 - 1/6 \\ 1/6 + x_1 \\ 2/3}_{\textstyle ,} \myvec{2x_1 - 1/6 \\ 1/6 + x_1 \\ 1/3 + x_1 \\ 1/6 + 2x_1} \right) : 1/6 < x_1 \leq 1/4 \right\}$ \\ \hline

 $\mathsf{G_3}$ & \smallpad{\begin{ytableau}[\mathsf]
    {03} & {24} & \mybar{14} & {04} \\
    \none & {12} & {01} & {34} \\
    \none & \none & {13} & {02}\\
    \none & \none & \none & {23}
\end{ytableau}} & $\left\{ \left( \myvec{4/20 \\ 5/20 \\ 7/20 \\ 13/20}_{\textstyle ,} \myvec{7/20 \\ 8/20 \\ 11/20 \\ 13/20} \right) \right\}$ \\ \hline

$\mathsf{G_4}$ & \smallpad{\begin{ytableau}[\mathsf]
    {12} & \mybar{04} & \mybar{14} & \mybar{03} \\
    \none & {24} & {13} & {02} \\
    \none & \none & {01} & {34}\\
    \none & \none & \none & {23}
\end{ytableau}} & $\left\{ \left( \myvec{2/20 \\ 7/20 \\ 8/20 \\ 11/20}_{\textstyle ,} \myvec{5/20 \\ 9/20 \\ 11/20 \\ 12/20} \right) \right\}$ \\ \hline
\end{longtable}

\begin{longtable}{|c|c|l|}
\hline
     & $\Delta \in \M(2,2)$ & \vpad{\mathbf{XY}\!_{2,2}(\Delta)} \\ \hline

$\mathsf{B_4}$ & \smallpad{\begin{ytableau}[\mathsf]
    {24} & \mybar{03} & {04} & \mybar{02} \\
    \none & {01} & {13} & {14} \\
    \none & \none & {12} & {23}\\
    \none & \none & \none & {34}
\end{ytableau}} & $\left\{ \left( \myvec{1/5 \\ x_2 \\ 1/5 + x_2 \\ 3/5}_{\textstyle ,} \myvec{3/5 - x_2 \\ 4/5 - x_2 \\ 3/5 \\ 1-x_2} \right) : 3/10 < x_2 < 2/5 \right\}$ \\ \hline

 $\mathsf{B_5}$ & \smallpad{\begin{ytableau}[\mathsf]
    {12} & \mybar{03} & {04} & \mybar{02} \\
    \none & {24} & {13} & {14} \\
    \none & \none & {01} & {23}\\
    \none & \none & \none & {34}
\end{ytableau}} & $\left\{ \left( \myvec{x_1 \\ 1/5 + x_1 \\ 2/5 + x_1 \\ 3/5}_{\textstyle ,} \myvec{1/5 \\ 3/5 - x_1 \\ 3/5 \\ 4/5 - x_1} \right) : 1/10 < x_1 < 1/5 \right\}$ \\ \hline

 $\mathsf{F_4} $ & \smallpad{\begin{ytableau}[\mathsf]
    {12} & {23} & {03} & \mybar{24} \\
    \none & {34} & \mybar{04} & {14} \\
    \none & \none & {02} & {13}\\
    \none & \none & \none & {01}
\end{ytableau}} & $\left\{ \left( \myvec{1/8 \\ x_2 \\ 1/4 + x_2 \\ 5/8}_{\textstyle ,} \myvec{x_2 - 1/8 \\ 1/4 \\ 1/4 + x_2 \\ 3/8 + x_2} \right) : 1/4 < x_2 < 3/8 \right\}$ \\ \hline
\end{longtable}

\bibliographystyle{amsplain}

\bibliography{References}

\end{document}